\documentclass[a4paper,11pt]{article}
\usepackage[centertags]{amsmath}
\usepackage{amsfonts}
\usepackage{amssymb}
\usepackage{amsthm}
\usepackage{newlfont}
\usepackage{graphicx}
\usepackage{t1enc}
\usepackage[latin1]{inputenc}
\usepackage[english]{babel}
\newlength{\defbaselineskip}
\newcommand{\setlinespacing}[1]%
           {\setlength{\baselineskip}{#1 \defbaselineskip}}
\newcommand{\bq }{\begin{equation}}
\newcommand{\eq }{\end{equation}}
\newcommand{\bbb }{\begin{eqnarray}}
\newcommand{\eee }{\end{eqnarray}}
\newcommand{\bb }{\begin{eqnarray*}}
\newcommand{\ee }{\end{eqnarray*}}

\newcommand{\ed }{\end{document}}
\theoremstyle{plain}
\newtheorem{theorem}{Theorem}[section]
\newtheorem{definition}{Definition}[section]
\newtheorem{example}{Example}[section]
\newtheorem{proposition}{Proposition}[section]
\newtheorem{corollary} {Corollary}[section]

\newtheorem{remark}{Remark}[section]
\newtheorem{lemma}{Lemma}[section]

\numberwithin{equation}{section}
\begin{document}
\begin{center}
\Large \textbf{Lyra Almost Ricci-Bourguignon Solitons on Twisted Warped Product Manifolds}
 \end{center}
\centerline{\bf Ayman Elsharkawy${}^{1*}$, Eman Ghareeb Rezk ${}^{2}$, Uday Chand De$^3$,} \centerline{\bf Emad F. Wanas$^4$ and Abdelrahman M. Tawfiq$^5$}
\centerline{\small${}^{1}$Department of Mathematics, Faculty of Science, Tanta University, Tanta, Egypt.}
\centerline{\small ayman\_ramadan@science.tanta.edu.eg}
\centerline{\small $^{2}$ Mathematical science department, College of science, Princess Nourah  }
\centerline{\small bint Abdlrahman University, P.O.Box 84428, Riyadh 11671, Saudi Arabia}
\centerline{\small  egrezk@pnu.edu.sa}
\centerline{\small${}^{3}$ Department of Pure Mathematics, University of Calcutta,} \centerline{\small 35 Ballygaunge Circular Road, Kolkata 700019, West Bengal, India.}
\centerline{\small uc\_de@yahoo.com}
\centerline{\small $^4$ Institute of Basic and Applied Science, College of Engineering and Technology,}\centerline{ Arab Academy for Science, Technology and Maritime Transport, Alexandria, Egypt}\centerline{ efw@aast.edu}
\centerline{\small${}^{5}$Department of Mathematics, Faculty of Education, Ain-Shams University, Cairo, Egypt.}
\centerline{abdelrhmanmagdi@edu.asu.edu.eg}

\begin{abstract}
	This paper defines Lyra almost Ricci--Bourguignon solitons on twisted warped product manifolds and investigates the interaction between the twisting function, the Lyra scale, and the soliton potential field. We derive the horizontal, vertical, and mixed components of the soliton equation and obtain trace, gradient, and factor-inheritance characterizations. The mixed equation yields new rigidity phenomena: a base-dependent Lyra scale preserves the ordinary separability obstruction, whereas a fiber-dependent scale admits a weighted compensation law that allows genuinely nonseparable twisting functions. We also establish Einstein reductions for conformal, Killing, homothetic, concurrent, and gradient potential fields. Exact flat and nonflat examples, together with local and compact nonexistence results, demonstrate the geometric significance of the scale--twisting interaction.
\end{abstract}



\section{Introduction}
\label{sec:introduction}

The Ricci flow introduced by Hamilton is a central evolution
equation in geometric analysis \cite{Hamilton1982}, while its
self-similar solutions connect geometric flows with Einstein
geometry \cite{Besse2008}. Bourguignon proposed a scalar-curvature
modification of the Ricci tensor \cite{Bourguignon1981}, leading to
the Ricci--Bourguignon flow studied in
\cite{CatinoEtAl2017}. Ricci--Bourguignon solitons and almost
solitons were subsequently investigated in \cite{Dwivedi2021}.
Related extensions include almost
\(\eta\)-Ricci--Bourguignon solitons \cite{BlagaTastan2021},
triviality criteria \cite{Ghosh2022}, and recent characterizations
based on the potential field and curvature
\cite{AlSodaisEtAl2025}.

Product manifolds provide a natural framework for decomposing
Einstein-type equations. Warped products were introduced by Bishop
and O'Neill \cite{BishopONeill1969}, and their semi-Riemannian
geometry was developed in \cite{ONeill1983}. A modern treatment of
their geometry and submanifolds is given in \cite{Chen2017}.
Twisted products extend this construction by allowing the twisting
function to depend on both factors \cite{PongeReckziegel1993},
whereas curvature conditions reducing them to warped products were
obtained in \cite{FernandezLopezEtAl2001}. Doubly warped products
constitute another important generalization \cite{Unal2001}.
Recent curvature studies on doubly and twisted products appear in
\cite{ElsharkawyEtAl2025Pseudo}, while mixed doubly sequential
warped products were examined in
\cite{ElsayiedTawfiqElsharkawy2025}.

Einstein-type solitons on product manifolds have received increasing
attention. Warped product \(\rho\)-Einstein solitons were studied in
\cite{BinTurkiEtAl2022}, and almost Ricci--Bourguignon solitons on
doubly warped products were considered in
\cite{ShenawyEtAl2023}. Sequential warped product versions were
investigated in \cite{PahanDutta2023} and
\cite{KayaOzgur2024}. Further gradient
\(\rho\)-Einstein results were obtained in
\cite{GulerUnal2025}, and doubly warped extensions were developed
in \cite{SemaryEtAl2026}. Related almost quasi-Yamabe structures on
twisted products were analyzed in
\cite{ElsharkawyEtAl2025Yamabe}. The closest ordinary counterpart
to the present work is the study of almost
Ricci--Bourguignon solitons and Einstein metrics on twisted warped
products \cite{ElsharkawyTawfiq2026}. More recently, the mixed
soliton equation was shown to impose strong separability and
rigidity restrictions on the twisting function
\cite{ElsharkawyCesaranoWanas2026}.
Lyra geometry enriches the metric structure by introducing an
additional positive scale function, following the original
construction of Lyra \cite{Lyra1951}. Its early scalar-tensor
interpretation was developed by Sen and Dunn
\cite{SenDunn1971}. A modern torsion-free and metric-compatible
scale formulation was established in
\cite{CuzinattoMoraisPimentel2021}. Recent exact models illustrate
the geometric influence of a nonconstant Lyra scale
\cite{BertinEtAl2025}, while its interpretation through local
transformations of length units was clarified in
\cite{ValadaoSobreroBergliaffa2026}.

To the best of our knowledge, almost Ricci--Bourguignon solitons
have not been systematically studied on twisted warped products
carrying a nonconstant Lyra scale. We therefore consider
\[
M=M_{1}\times_{f}M_{2},
\qquad
g=g_{1}\oplus f^{2}g_{2},
\qquad
g_{L}=\varphi^{2}g,
\]
and write
\[
k=\ln f,
\qquad
\sigma=\ln\varphi.
\]
The central question is whether the Lyra scale preserves the
ordinary separability obstruction or compensates for the mixed
derivatives of a genuinely nonseparable twisting function.

The main results are as follows. We derive the complete horizontal,
vertical, and mixed decomposition of the Lyra almost
Ricci--Bourguignon soliton equation and obtain trace and gradient
identities. We establish factor-inheritance criteria through
explicit trace-free obstruction tensors. We prove that a
base-dependent scale preserves the ordinary multiplicative
separability condition, whereas a fiber-dependent scale admits a
new weighted compensation law and genuinely twisted local models.
For separated data, we obtain the rigidity relation
\[
f_{1}\varphi_{1}=\mathrm{constant}.
\]
We also characterize Einstein Lyra metrics through conformal
potential fields, including Killing, homothetic, concurrent, and
gradient cases. Finally, exact examples and local and compact
nonexistence results demonstrate the geometric significance of the
scale--twisting interaction.

Section~\ref{sec:lyra-twisted-geometry} establishes the required
Lyra twisted product geometry.
Section~\ref{sec:lyra-arbs} develops the soliton equations.
Section~\ref{sec:factor-inheritance-rigidity} treats factor
inheritance and rigidity, while
Section~\ref{sec:einstein-special-fields} studies Einstein
reductions and special potential fields.
Section~\ref{sec:examples-nonexistence} presents examples and
nonexistence applications.

\section{Lyra Twisted Warped Product Geometry}
\label{sec:lyra-twisted-geometry}

This section fixes the geometric notation and records the identities
required in the subsequent soliton analysis. The standard theory of
warped products originates from Bishop and O'Neill
\cite{BishopONeill1969}, while its semi-Riemannian formulation is
developed in \cite{ONeill1983}. Twisted products were systematically
studied by Ponge and Reckziegel \cite{PongeReckziegel1993}. For the
Lyra structure, we adopt the torsion-free and metric-compatible
scale representation introduced in
\cite{CuzinattoMoraisPimentel2021} and developed further in
\cite{ValadaoSobreroBergliaffa2026}.

\subsection{Twisted warped product structure}
\label{subsec:twisted-product-structure}

Let $(M_{1},g_{1})$ and $(M_{2},g_{2})
$
be pseudo-Riemannian manifolds of dimensions \(n_{1}\) and
\(n_{2}\), respectively, and set
$
n=n_{1}+n_{2}.
$
Let
$
f:M_{1}\times M_{2}\longrightarrow (0,\infty)
$
be a smooth positive function. The twisted warped product
$
M=M_{1}\times_{f}M_{2}
$
is the product manifold \(M_{1}\times M_{2}\) equipped with the
metric
\begin{equation}
	g
	=
	g_{1}\oplus f^{2}g_{2}.
	\label{eq:twisted-product-metric}
\end{equation}
Throughout the paper, we use the logarithmic twisting function
\begin{equation}
	k=\ln f.
	\label{eq:logarithmic-twisting-function}
\end{equation}

The tangent bundle admits the orthogonal decomposition
\begin{equation}
	TM
	=
	\mathcal H\oplus\mathcal V,
	\label{eq:horizontal-vertical-decomposition}
\end{equation}
where
\[
\mathcal H\simeq TM_{1},
\qquad
\mathcal V\simeq TM_{2}.
\]
Vector fields
$
X,Y\in\Gamma(\mathcal H),
\qquad
U,V\in\Gamma(\mathcal V)
$
will always denote the canonical horizontal and vertical lifts,
respectively. In particular,
\[
[X,U]=0.
\]

If \(f\) depends only on \(M_{1}\), then
\eqref{eq:twisted-product-metric} reduces to the ordinary warped
product metric. If \(f\) is constant, the metric becomes the direct
product metric. Conditions under which a twisted product reduces
to a warped product were investigated in
\cite{FernandezLopezEtAl2001}.

For a smooth function \(h\in C^{\infty}(M)\), we denote by
$\{\nabla_{1}h,\nabla_{2}h,\operatorname{Hess}_{1}h, \operatorname{Hess}_{2}h, \Delta_{1}h,\Delta_{2}h\}$
the partial gradients, Hessians, and Laplacians computed on the
corresponding factor, with the other variable regarded as a
parameter. The gradient of \(h\) with respect to \(g\) is
\begin{equation}
	\nabla h
	=
	\nabla_{1}h
	+
	f^{-2}\nabla_{2}h,
	\label{eq:twisted-gradient}
\end{equation}
and consequently
\begin{equation}
	\lvert\nabla h\rvert_{g}^{2}
	=
	\lvert\nabla_{1}h\rvert_{g_{1}}^{2}
	+
	f^{-2}\lvert\nabla_{2}h\rvert_{g_{2}}^{2}.
	\label{eq:twisted-gradient-norm}
\end{equation}

Let \(\nabla\), \(\nabla^{1}\), and \(\nabla^{2}\) denote the
Levi-Civita connections of \(g\), \(g_{1}\), and \(g_{2}\),
respectively. The standard connection identities on a twisted
warped product are
\begin{subequations}
	\label{eq:twisted-connection}
	\begin{align}
		\nabla_{X}Y
		&=
		\nabla^{1}_{X}Y,
		\label{eq:twisted-connection-horizontal}
		\\
		\nabla_{X}U
		=
		\nabla_{U}X
		&=
		X(k)U,
		\label{eq:twisted-connection-mixed}
		\\
		\nabla_{U}V
		&=
		\nabla^{2}_{U}V
		+
		U(k)V
		+
		V(k)U
		-
		g_{2}(U,V)\nabla_{2}k
		-
		f^{2}g_{2}(U,V)\nabla_{1}k.
		\label{eq:twisted-connection-vertical}
	\end{align}
\end{subequations}

Using \eqref{eq:twisted-connection}, the Hessian of
\(h\in C^{\infty}(M)\) satisfies
\begin{subequations}
	\label{eq:twisted-hessian}
	\begin{align}
		\operatorname{Hess}h(X,Y)
		&=
		\operatorname{Hess}_{1}h(X,Y),
		\label{eq:twisted-hessian-horizontal}
		\\
		\operatorname{Hess}h(X,U)
		&=
		XU(h)-X(k)U(h),
		\label{eq:twisted-hessian-mixed}
		\\
		\operatorname{Hess}h(U,V)
		&=
		\operatorname{Hess}_{2}h(U,V)
		-
		U(k)V(h)
		-
		V(k)U(h)
		\nonumber\\
		&\quad
		+
		g_{2}(U,V)
		\left\langle\nabla_{2}k,\nabla_{2}h\right\rangle_{g_{2}}
		\nonumber\\
		&\quad
		+
		f^{2}g_{2}(U,V)
		\left\langle\nabla_{1}k,\nabla_{1}h\right\rangle_{g_{1}}.
		\label{eq:twisted-hessian-vertical}
	\end{align}
\end{subequations}

Taking the trace of \eqref{eq:twisted-hessian} gives
\begin{equation}
	\Delta h
	=
	\Delta_{1}h
	+
	n_{2}
	\left\langle\nabla_{1}k,\nabla_{1}h\right\rangle_{g_{1}}
	+
	f^{-2}
	\left[
	\Delta_{2}h
	+
	(n_{2}-2)
	\left\langle\nabla_{2}k,\nabla_{2}h\right\rangle_{g_{2}}
	\right].
	\label{eq:twisted-laplacian}
\end{equation}

\subsection{Lyra scale structure}
\label{subsec:lyra-scale-structure}

Following the scale-metric formulation of Lyra geometry, let
\[
\varphi:M\longrightarrow(0,\infty)
\]
be a smooth positive scale function and introduce
\begin{equation}
	\sigma=\ln\varphi.
	\label{eq:logarithmic-lyra-scale}
\end{equation}
The associated Lyra metric is defined by
\begin{equation}
	g_{L}
	=
	\varphi^{2}g
	=
	e^{2\sigma}
	\left(
	g_{1}\oplus f^{2}g_{2}
	\right).
	\label{eq:lyra-twisted-metric}
\end{equation}

The triple
\begin{equation}
	\left(
	M_{1}\times_{f}M_{2},
	g,
	\varphi
	\right)
	\label{eq:lyra-twisted-product}
\end{equation}
will be called a \emph{Lyra twisted warped product}, and the
geometry determined by \(g_{L}\) will be referred to as its Lyra
geometry. When \(\varphi\) is constant, the Lyra structure reduces,
up to a constant homothety, to the ordinary twisted warped product.

Let \(\nabla^{L}\) denote the Levi-Civita connection of \(g_{L}\).
The standard conformal connection identity gives
\begin{equation}
	\nabla^{L}_{A}B
	=
	\nabla_{A}B
	+
	A(\sigma)B
	+
	B(\sigma)A
	-
	g(A,B)\nabla\sigma,
	\label{eq:lyra-connection-general}
\end{equation}
for arbitrary \(A,B\in\Gamma(TM)\).

For later use, define the combined scale-twisting function
\begin{equation}
	\omega
	=
	k+\sigma
	=
	\ln(f\varphi).
	\label{eq:combined-scale-twisting-function}
\end{equation}
Using \eqref{eq:twisted-connection} in
\eqref{eq:lyra-connection-general}, one obtains
\begin{subequations}
	\label{eq:lyra-connection-components}
	\begin{align}
		\nabla^{L}_{X}Y
		&=
		\nabla^{1}_{X}Y
		+
		X(\sigma)Y
		+
		Y(\sigma)X
		-
		g_{1}(X,Y)
		\left(
		\nabla_{1}\sigma
		+
		f^{-2}\nabla_{2}\sigma
		\right),
		\label{eq:lyra-connection-horizontal}
		\\
		\nabla^{L}_{X}U
		=
		\nabla^{L}_{U}X
		&=
		X(\omega)U
		+
		U(\sigma)X,
		\label{eq:lyra-connection-mixed}
		\\
		\nabla^{L}_{U}V
		&=
		\nabla^{2}_{U}V
		+
		U(\omega)V
		+
		V(\omega)U
		\nonumber\\
		&\quad
		-
		g_{2}(U,V)\nabla_{2}\omega
		-
		f^{2}g_{2}(U,V)\nabla_{1}\omega.
		\label{eq:lyra-connection-vertical}
	\end{align}
\end{subequations}

Equations \eqref{eq:lyra-connection-horizontal}--
\eqref{eq:lyra-connection-vertical} show that both \(f\) and
\(\varphi\) influence the Lyra connection. In particular, their
vertical contributions occur through the combined function
\(\omega=\ln(f\varphi)\), whereas the horizontal connection also
contains the independent variation of the Lyra scale.

The volume elements of \(g\) and \(g_{L}\) are related by
\begin{equation}
	d\mu_{g}
	=
	f^{n_{2}}\,
	d\mu_{g_{1}}\,
	d\mu_{g_{2}},
	\qquad
	d\mu_{L}
	=
	\varphi^{n}f^{n_{2}}\,
	d\mu_{g_{1}}\,
	d\mu_{g_{2}}.
	\label{eq:lyra-volume-element}
\end{equation}

\subsection{Curvature identities}
\label{subsec:curvature-identities}

Let
\[
\operatorname{Ric},
\qquad
R
\]
denote the Ricci tensor and scalar curvature of \(g\). The Ricci
tensor of the twisted warped product
\eqref{eq:twisted-product-metric} is determined by
\begin{subequations}
	\label{eq:twisted-ricci-components}
	\begin{align}
		\operatorname{Ric}(X,Y)
		&=
		\operatorname{Ric}_{1}(X,Y)
		-
		n_{2}
		\left[
		\operatorname{Hess}_{1}k(X,Y)
		+
		X(k)Y(k)
		\right],
		\label{eq:twisted-ricci-horizontal}
		\\
		\operatorname{Ric}(X,U)
		&=
		-
		(n_{2}-1)XU(k),
		\label{eq:twisted-ricci-mixed}
		\\
		\operatorname{Ric}(U,V)
		&=
		\operatorname{Ric}_{2}(U,V)
		-
		(n_{2}-2)
		\left[
		\operatorname{Hess}_{2}k(U,V)
		-
		U(k)V(k)
		\right]
		\nonumber\\
		&\quad
		-
		\left[
		\Delta_{2}k
		+
		(n_{2}-2)
		\lvert\nabla_{2}k\rvert_{g_{2}}^{2}
		\right]g_{2}(U,V)
		\nonumber\\
		&\quad
		-
		f^{2}
		\left[
		\Delta_{1}k
		+
		n_{2}
		\lvert\nabla_{1}k\rvert_{g_{1}}^{2}
		\right]g_{2}(U,V).
		\label{eq:twisted-ricci-vertical}
	\end{align}
\end{subequations}

The corresponding scalar curvature is
\begin{align}
	R
	={}&
	R_{1}
	+
	f^{-2}R_{2}
	-
	2n_{2}\Delta_{1}k
	-
	n_{2}(n_{2}+1)
	\lvert\nabla_{1}k\rvert_{g_{1}}^{2}
	\nonumber\\
	&-
	f^{-2}
	\left[
	2(n_{2}-1)\Delta_{2}k
	+
	(n_{2}-1)(n_{2}-2)
	\lvert\nabla_{2}k\rvert_{g_{2}}^{2}
	\right].
	\label{eq:twisted-scalar-curvature}
\end{align}

Let
\[
\operatorname{Ric}^{L},
\qquad
R^{L}
\]
denote the Ricci tensor and scalar curvature of the Lyra metric
\(g_{L}\). The conformal transformation formulas
\cite{Besse2008} yield
\begin{align}
	\operatorname{Ric}^{L}
	={}&
	\operatorname{Ric}
	-
	(n-2)
	\left(
	\operatorname{Hess}\sigma
	-
	d\sigma\otimes d\sigma
	\right)
	\nonumber\\
	&-
	\left[
	\Delta\sigma
	+
	(n-2)\lvert\nabla\sigma\rvert_{g}^{2}
	\right]g,
	\label{eq:lyra-ricci-transformation}
\end{align}
and
\begin{equation}
	R^{L}
	=
	e^{-2\sigma}
	\left[
	R
	-
	2(n-1)\Delta\sigma
	-
	(n-1)(n-2)
	\lvert\nabla\sigma\rvert_{g}^{2}
	\right].
	\label{eq:lyra-scalar-transformation}
\end{equation}
Here \(R\), \(\Delta\sigma\), and
\(\lvert\nabla\sigma\rvert_{g}^{2}\) are given by
\eqref{eq:twisted-scalar-curvature},
\eqref{eq:twisted-laplacian}, and
\eqref{eq:twisted-gradient-norm}, respectively.

In particular, using \eqref{eq:twisted-hessian-mixed} and
\eqref{eq:twisted-ricci-mixed} in
\eqref{eq:lyra-ricci-transformation}, the mixed Lyra Ricci
component is
\begin{align}
	\operatorname{Ric}^{L}(X,U)
	={}&
	-
	(n_{2}-1)XU(k)
	\nonumber\\
	&-
	(n-2)
	\left[
	XU(\sigma)
	-
	X(k)U(\sigma)
	-
	X(\sigma)U(\sigma)
	\right].
	\label{eq:mixed-lyra-ricci}
\end{align}
Equation \eqref{eq:mixed-lyra-ricci} will play a central role in
the scale-twisting rigidity analysis.

\subsection{Lyra differential identities}
\label{subsec:lyra-differential-identities}

Let \(W\in\Gamma(TM)\). Since \(g_{L}=e^{2\sigma}g\), the Lie
derivative satisfies
\begin{equation}
	\mathcal L_{W}g_{L}
	=
	e^{2\sigma}
	\left[
	\mathcal L_{W}g
	+
	2W(\sigma)g
	\right].
	\label{eq:lyra-lie-derivative}
\end{equation}

For any smooth function \(u\in C^{\infty}(M)\), its Lyra gradient
is
\begin{equation}
	\nabla^{L}u
	=
	e^{-2\sigma}\nabla u.
	\label{eq:lyra-gradient-function}
\end{equation}
The Lyra Hessian is
\begin{align}
	\operatorname{Hess}^{L}u
	={}&
	\operatorname{Hess}u
	-
	du\otimes d\sigma
	-
	d\sigma\otimes du
	\nonumber\\
	&+
	\left\langle\nabla u,\nabla\sigma\right\rangle_{g}g,
	\label{eq:lyra-hessian-function}
\end{align}
where the components of \(\operatorname{Hess}u\) are given in
\eqref{eq:twisted-hessian}. Taking the Lyra trace gives
\begin{equation}
	\Delta^{L}u
	=
	e^{-2\sigma}
	\left[
	\Delta u
	+
	(n-2)
	\left\langle\nabla\sigma,\nabla u\right\rangle_{g}
	\right].
	\label{eq:lyra-laplacian-function}
\end{equation}

The identities collected in
\eqref{eq:twisted-connection}--
\eqref{eq:lyra-laplacian-function} will be used throughout the
paper. In particular, the horizontal, vertical, and mixed
components of the Lyra almost Ricci--Bourguignon soliton equation
will be obtained directly from
\eqref{eq:twisted-ricci-components},
\eqref{eq:lyra-ricci-transformation}, and
\eqref{eq:lyra-lie-derivative}, without repeating the preceding
geometric computations.

\section{Lyra Almost Ricci--Bourguignon Solitons}
\label{sec:lyra-arbs}

Almost Ricci--Bourguignon solitons generalize Ricci solitons by
including the scalar-curvature contribution associated with the
Ricci--Bourguignon flow; see
\cite{CatinoEtAl2017,Dwivedi2021}. Their behavior on warped,
doubly warped, sequential warped, and twisted warped product
manifolds has been investigated in
\cite{BinTurkiEtAl2022,ShenawyEtAl2023,PahanDutta2023,
	KayaOzgur2024,ElsharkawyTawfiq2026}. In this section, we define
the corresponding structure for the Lyra twisted warped product
defined in Section~\ref{sec:lyra-twisted-geometry} and derive its
complete horizontal, vertical, and mixed characterization.

Throughout this section, let
\[
M=M_{1}\times_{f}M_{2},
\qquad
g=g_{1}\oplus f^{2}g_{2},
\qquad
g_{L}=e^{2\sigma}g,
\]
where
\[
k=\ln f,
\qquad
\sigma=\ln\varphi,
\qquad
\omega=k+\sigma.
\]
We write
\[
n_{1}=\dim M_{1},
\qquad
n_{2}=\dim M_{2},
\qquad
n=n_{1}+n_{2}.
\]

\subsection{Definition and elementary reductions}
\label{subsec:lyra-arbs-definition}

\begin{definition}
	\label{def:lyra-arbs}
	A quadruple
	\[
	(g_{L},Z,\lambda,\rho)
	\]
	on a Lyra twisted warped product \(M\) is called a
	\emph{Lyra almost Ricci--Bourguignon soliton} if
	\begin{equation}
		\operatorname{Ric}^{L}
		+
		\frac{1}{2}\mathcal L_{Z}g_{L}
		=
		\left(
		\lambda+\rho R^{L}
		\right)g_{L},
		\label{eq:lyra-arbs-equation}
	\end{equation}
	where
	\[
	Z\in\Gamma(TM),
	\qquad
	\lambda\in C^{\infty}(M),
	\qquad
	\rho\in\mathbb R.
	\]
	The vector field \(Z\) is called the \emph{potential vector
		field}, while \(\lambda\) is called the \emph{soliton function}.
\end{definition}

For notational convenience, set
\begin{equation}
	\Lambda
	=
	\lambda+\rho R^{L}.
	\label{eq:lyra-arbs-lambda}
\end{equation}
Accordingly, Equation~\eqref{eq:lyra-arbs-equation} becomes
\begin{equation}
	\operatorname{Ric}^{L}
	+
	\frac{1}{2}\mathcal L_{Z}g_{L}
	=
	\Lambda g_{L}.
	\label{eq:lyra-arbs-compact}
\end{equation}

When \(\lambda\) is constant, the structure is called a
\emph{Lyra Ricci--Bourguignon soliton}. When \(\rho=0\),
Equation~\eqref{eq:lyra-arbs-equation} reduces to a Lyra almost
Ricci soliton. If \(\sigma=0\), then \(g_{L}=g\), and
\eqref{eq:lyra-arbs-equation} reduces to the ordinary almost
Ricci--Bourguignon soliton equation on a twisted warped product.

When the soliton function has a definite sign, the soliton is
called shrinking, steady, or expanding according as
\[
\lambda>0,
\qquad
\lambda=0,
\qquad
\lambda<0,
\]
respectively.

\subsection{Splitting of the Lyra Lie derivative}
\label{subsec:lyra-lie-splitting}

For the factorwise characterization, we assume that the potential
field is projectable and admits the decomposition
\begin{equation}
	Z
	=
	Z_{1}+Z_{2},
	\label{eq:projectable-potential}
\end{equation}
where \(Z_{1}\) and \(Z_{2}\) are the canonical lifts of vector
fields on \(M_{1}\) and \(M_{2}\), respectively.

\begin{proposition}
	\label{prop:lyra-lie-splitting}
	Let \(Z=Z_{1}+Z_{2}\) be as in
	\eqref{eq:projectable-potential}. For
	\[
	X,Y\in\Gamma(\mathcal H),
	\qquad
	U,V\in\Gamma(\mathcal V),
	\]
	the Lie derivative of the Lyra metric satisfies
	\begin{subequations}
		\label{eq:lyra-lie-components}
		\begin{align}
			(\mathcal L_{Z}g_{L})(X,Y)
			&=
			e^{2\sigma}
			\left[
			(\mathcal L_{Z_{1}}g_{1})(X,Y)
			+
			2Z(\sigma)g_{1}(X,Y)
			\right],
			\label{eq:lyra-lie-horizontal}
			\\
			(\mathcal L_{Z}g_{L})(X,U)
			&=
			0,
			\label{eq:lyra-lie-mixed}
			\\
			(\mathcal L_{Z}g_{L})(U,V)
			&=
			e^{2\sigma}f^{2}
			\left[
			(\mathcal L_{Z_{2}}g_{2})(U,V)
			+
			2Z(\omega)g_{2}(U,V)
			\right].
			\label{eq:lyra-lie-vertical}
		\end{align}
	\end{subequations}
\end{proposition}

\begin{proof}
	Using Equation~\eqref{eq:lyra-lie-derivative}, we have
	\[
	\mathcal L_{Z}g_{L}
	=
	e^{2\sigma}
	\left(
	\mathcal L_{Z}g
	+
	2Z(\sigma)g
	\right).
	\]
	Since \(Z_{1}\) and \(Z_{2}\) are projectable and
	\(g=g_{1}\oplus f^{2}g_{2}\), the horizontal component is
	\[
	(\mathcal L_{Z}g)(X,Y)
	=
	(\mathcal L_{Z_{1}}g_{1})(X,Y).
	\]
	The orthogonality of the horizontal and vertical distributions,
	together with the projectability of the fields, gives
	\[
	(\mathcal L_{Z}g)(X,U)=0.
	\]
	For vertical fields,
	\[
	(\mathcal L_{Z}g)(U,V)
	=
	f^{2}
	\left[
	(\mathcal L_{Z_{2}}g_{2})(U,V)
	+
	2Z(k)g_{2}(U,V)
	\right].
	\]
	Combining these identities with
	\(Z(\omega)=Z(k)+Z(\sigma)\) proves
	\eqref{eq:lyra-lie-components}.
\end{proof}

\subsection{Component characterization}
\label{subsec:lyra-arbs-components}

Define
\begin{equation}
	\Theta_{\sigma}
	=
	\Delta\sigma
	+
	(n-2)
	\lvert\nabla\sigma\rvert_{g}^{2}.
	\label{eq:theta-sigma}
\end{equation}
The quantities in \eqref{eq:theta-sigma} are explicitly determined
by Equations~\eqref{eq:twisted-laplacian} and
\eqref{eq:twisted-gradient-norm}.

\begin{theorem}[Horizontal, vertical, and mixed soliton system]
	\label{thm:lyra-arbs-component-characterization}
	Let
	\[
	(M,g_{L},Z,\lambda,\rho)
	\]
	be a Lyra twisted warped product equipped with the projectable
	field \(Z=Z_{1}+Z_{2}\). Then
	\eqref{eq:lyra-arbs-equation} holds if and only if the following
	three systems are satisfied.
	
	\medskip
	\noindent
	\textnormal{\textbf{(i) Horizontal equation.}}
	For all \(X,Y\in\Gamma(\mathcal H)\),
	\begin{align}
		&\operatorname{Ric}_{1}(X,Y)
		-
		n_{2}
		\left[
		\operatorname{Hess}_{1}k(X,Y)
		+
		X(k)Y(k)
		\right]
		\nonumber\\
		&\quad
		-
		(n-2)
		\left[
		\operatorname{Hess}_{1}\sigma(X,Y)
		-
		X(\sigma)Y(\sigma)
		\right]
		+
		\frac{e^{2\sigma}}{2}
		(\mathcal L_{Z_{1}}g_{1})(X,Y)
		\nonumber\\
		&=
		\left[
		e^{2\sigma}
		\bigl(
		\Lambda-Z(\sigma)
		\bigr)
		+
		\Theta_{\sigma}
		\right]
		g_{1}(X,Y).
		\label{eq:lyra-arbs-horizontal}
	\end{align}
	
	\medskip
	\noindent
	\textnormal{\textbf{(ii) Vertical equation.}}
	For all \(U,V\in\Gamma(\mathcal V)\),
	\begin{align}
		&\operatorname{Ric}(U,V)
		-
		(n-2)
		\left[
		\operatorname{Hess}\sigma(U,V)
		-
		U(\sigma)V(\sigma)
		\right]
		\nonumber\\
		&\quad
		+
		\frac{e^{2\sigma}f^{2}}{2}
		(\mathcal L_{Z_{2}}g_{2})(U,V)
		\nonumber\\
		&=
		f^{2}
		\left[
		e^{2\sigma}
		\bigl(
		\Lambda-Z(\omega)
		\bigr)
		+
		\Theta_{\sigma}
		\right]
		g_{2}(U,V).
		\label{eq:lyra-arbs-vertical}
	\end{align}
	Here \(\operatorname{Ric}(U,V)\) and
	\(\operatorname{Hess}\sigma(U,V)\) are given by
	\eqref{eq:twisted-ricci-vertical} and
	\eqref{eq:twisted-hessian-vertical}, respectively.
	
	\medskip
	\noindent
	\textnormal{\textbf{(iii) Mixed equation.}}
	For all
	\(X\in\Gamma(\mathcal H)\) and
	\(U\in\Gamma(\mathcal V)\),
	\begin{align}
		&(n_{2}-1)XU(k)
		\nonumber\\
		&\quad
		+
		(n-2)
		\left[
		XU(\sigma)
		-
		X(k)U(\sigma)
		-
		X(\sigma)U(\sigma)
		\right]
		=
		0.
		\label{eq:lyra-arbs-mixed}
	\end{align}
\end{theorem}

\begin{proof}
	Evaluate Equation~\eqref{eq:lyra-arbs-compact} on horizontal
	fields \(X\) and \(Y\). Using
	\eqref{eq:twisted-ricci-horizontal} in
	\eqref{eq:lyra-ricci-transformation}, we obtain
	\begin{align*}
		\operatorname{Ric}^{L}(X,Y)
		={}&
		\operatorname{Ric}_{1}(X,Y)
		-
		n_{2}
		\left[
		\operatorname{Hess}_{1}k(X,Y)
		+
		X(k)Y(k)
		\right]
		\\
		&-
		(n-2)
		\left[
		\operatorname{Hess}_{1}\sigma(X,Y)
		-
		X(\sigma)Y(\sigma)
		\right]
		-
		\Theta_{\sigma}g_{1}(X,Y).
	\end{align*}
	By Equation~\eqref{eq:lyra-lie-horizontal},
	\[
	\frac{1}{2}
	(\mathcal L_{Z}g_{L})(X,Y)
	=
	\frac{e^{2\sigma}}{2}
	(\mathcal L_{Z_{1}}g_{1})(X,Y)
	+
	e^{2\sigma}Z(\sigma)g_{1}(X,Y).
	\]
	Since
	\[
	g_{L}(X,Y)
	=
	e^{2\sigma}g_{1}(X,Y),
	\]
	substitution into \eqref{eq:lyra-arbs-compact} gives
	\eqref{eq:lyra-arbs-horizontal}.
	
	For vertical fields, Equation~\eqref{eq:lyra-ricci-transformation}
	gives
	\begin{align*}
		\operatorname{Ric}^{L}(U,V)
		={}&
		\operatorname{Ric}(U,V)
		-
		(n-2)
		\left[
		\operatorname{Hess}\sigma(U,V)
		-
		U(\sigma)V(\sigma)
		\right]
		\\
		&-
		f^{2}\Theta_{\sigma}g_{2}(U,V).
	\end{align*}
	Using \eqref{eq:lyra-lie-vertical} and
	\[
	g_{L}(U,V)
	=
	e^{2\sigma}f^{2}g_{2}(U,V),
	\]
	Equation~\eqref{eq:lyra-arbs-compact} becomes precisely
	\eqref{eq:lyra-arbs-vertical}.
	
	Finally, since
	\[
	g_{L}(X,U)=0
	\]
	and the mixed Lie derivative vanishes by
	\eqref{eq:lyra-lie-mixed}, the mixed part of the soliton equation
	is equivalent to
	\[
	\operatorname{Ric}^{L}(X,U)=0.
	\]
	Using Equation~\eqref{eq:mixed-lyra-ricci} yields
	\eqref{eq:lyra-arbs-mixed}.
	
	Conversely, every pair of tangent vectors admits a unique
	horizontal--vertical decomposition. Therefore,
	\eqref{eq:lyra-arbs-horizontal},
	\eqref{eq:lyra-arbs-vertical}, and
	\eqref{eq:lyra-arbs-mixed} together imply
	\eqref{eq:lyra-arbs-compact} on all of \(TM\).
\end{proof}

\begin{remark}
	\label{rem:mixed-scale-twisting}
	Equation~\eqref{eq:lyra-arbs-mixed} is independent of the
	soliton function \(\lambda\), the parameter \(\rho\), and the
	projectable potential field \(Z\). It is therefore a purely
	geometric compatibility condition between the twisting function
	\(f\) and the Lyra scale \(\varphi\).
	
	If \(\sigma=0\), then
	\eqref{eq:lyra-arbs-mixed} reduces to
	\[
	(n_{2}-1)XU(k)=0.
	\]
	Thus, when \(n_{2}>1\), the ordinary twisted-product obstruction
	\[
	XU(k)=0
	\]
	is recovered. The additional \(\sigma\)-terms in
	\eqref{eq:lyra-arbs-mixed} are the source of the
	scale--twisting compensation studied in
	Section~\ref{sec:factor-inheritance-rigidity}.
\end{remark}

\subsection{Trace and divergence identities}
\label{subsec:lyra-arbs-trace}

\begin{proposition}[Trace identity]
	\label{prop:lyra-arbs-trace}
	Every Lyra almost Ricci--Bourguignon soliton satisfies
	\begin{equation}
		(1-n\rho)R^{L}
		+
		\operatorname{div}_{L}Z
		=
		n\lambda.
		\label{eq:lyra-arbs-trace}
	\end{equation}
	If \(Z=Z_{1}+Z_{2}\) is projectable, then
	\begin{align}
		\operatorname{div}_{L}Z
		={}&
		\operatorname{div}_{1}Z_{1}
		+
		\operatorname{div}_{2}Z_{2}
		+
		n_{2}Z(k)
		+
		nZ(\sigma).
		\label{eq:lyra-projectable-divergence}
	\end{align}
	Hence,
	\begin{align}
		(1-n\rho)R^{L}
		&+
		\operatorname{div}_{1}Z_{1}
		+
		\operatorname{div}_{2}Z_{2}
		\nonumber\\
		&+
		n_{2}Z(k)
		+
		nZ(\sigma)
		=
		n\lambda.
		\label{eq:lyra-arbs-projectable-trace}
	\end{align}
\end{proposition}

\begin{proof}
	Taking the trace of
	\eqref{eq:lyra-arbs-equation} with respect to \(g_{L}\) gives
	\[
	R^{L}
	+
	\frac{1}{2}
	\operatorname{tr}_{g_{L}}
	(\mathcal L_{Z}g_{L})
	=
	n\lambda+n\rho R^{L}.
	\]
	Since
	\[
	\frac{1}{2}
	\operatorname{tr}_{g_{L}}
	(\mathcal L_{Z}g_{L})
	=
	\operatorname{div}_{L}Z,
	\]
	we obtain \eqref{eq:lyra-arbs-trace}.
	
	By Equation~\eqref{eq:lyra-volume-element},
	\[
	d\mu_{L}
	=
	e^{n\sigma}f^{n_{2}}\,
	d\mu_{g_{1}}d\mu_{g_{2}}.
	\]
	Therefore, for a projectable field \(Z=Z_{1}+Z_{2}\),
	\[
	\operatorname{div}_{L}Z
	=
	\operatorname{div}_{1}Z_{1}
	+
	\operatorname{div}_{2}Z_{2}
	+
	Z\!\left(
	\ln(e^{n\sigma}f^{n_{2}})
	\right).
	\]
	Using
	\[
	Z\!\left(
	\ln(e^{n\sigma}f^{n_{2}})
	\right)
	=
	nZ(\sigma)+n_{2}Z(k),
	\]
	we obtain \eqref{eq:lyra-projectable-divergence}. Substitution
	into \eqref{eq:lyra-arbs-trace} proves
	\eqref{eq:lyra-arbs-projectable-trace}.
\end{proof}

\begin{remark}
	\label{rem:critical-bourguignon-parameter}
	The value
	\[
	\rho=\frac{1}{n}
	\]
	is exceptional. In this case, the scalar-curvature term
	disappears from \eqref{eq:lyra-arbs-trace}, and every soliton must
	satisfy
	\[
	\operatorname{div}_{L}Z=n\lambda.
	\]
\end{remark}

\subsection{Gradient Lyra almost Ricci--Bourguignon solitons}
\label{subsec:gradient-lyra-arbs}

\begin{definition}
	\label{def:gradient-lyra-arbs}
	A Lyra almost Ricci--Bourguignon soliton is called
	\emph{gradient} if there exists a smooth function
	\(u\in C^{\infty}(M)\) such that
	\[
	Z=\nabla^{L}u.
	\]
	In this case, Equation~\eqref{eq:lyra-arbs-equation} becomes
	\begin{equation}
		\operatorname{Ric}^{L}
		+
		\operatorname{Hess}^{L}u
		=
		\left(
		\lambda+\rho R^{L}
		\right)g_{L}.
		\label{eq:gradient-lyra-arbs}
	\end{equation}
\end{definition}

Define
\begin{equation}
	\Psi_{u,\sigma}
	=
	\left\langle
	\nabla u,\nabla\sigma
	\right\rangle_{g}
	-
	\Delta\sigma
	-
	(n-2)
	\lvert\nabla\sigma\rvert_{g}^{2}.
	\label{eq:psi-u-sigma}
\end{equation}

\begin{theorem}[Conformal form of the gradient equation]
	\label{thm:gradient-lyra-arbs-conformal-form}
	The gradient Lyra almost Ricci--Bourguignon equation
	\eqref{eq:gradient-lyra-arbs} is equivalent to
	\begin{align}
		\operatorname{Ric}
		+
		\operatorname{Hess}u
		&-
		(n-2)\operatorname{Hess}\sigma
		+
		(n-2)d\sigma\otimes d\sigma
		\nonumber\\
		&-
		du\otimes d\sigma
		-
		d\sigma\otimes du
		+
		\Psi_{u,\sigma}g
		\nonumber\\
		&=
		e^{2\sigma}
		\left(
		\lambda+\rho R^{L}
		\right)g.
		\label{eq:gradient-lyra-arbs-conformal}
	\end{align}
\end{theorem}

\begin{proof}
	Using Equations~\eqref{eq:lyra-ricci-transformation} and
	\eqref{eq:lyra-hessian-function}, we have
	\begin{align*}
		\operatorname{Ric}^{L}
		+
		\operatorname{Hess}^{L}u
		={}&
		\operatorname{Ric}
		-
		(n-2)
		\left(
		\operatorname{Hess}\sigma
		-
		d\sigma\otimes d\sigma
		\right)
		\\
		&-
		\left[
		\Delta\sigma
		+
		(n-2)
		\lvert\nabla\sigma\rvert_{g}^{2}
		\right]g
		\\
		&+
		\operatorname{Hess}u
		-
		du\otimes d\sigma
		-
		d\sigma\otimes du
		\\
		&+
		\left\langle
		\nabla u,\nabla\sigma
		\right\rangle_{g}g.
	\end{align*}
	Collecting the terms proportional to \(g\) gives
	\(\Psi_{u,\sigma}g\). Since
	\[
	g_{L}=e^{2\sigma}g,
	\]
	Equation~\eqref{eq:gradient-lyra-arbs} is equivalent to
	\eqref{eq:gradient-lyra-arbs-conformal}.
\end{proof}

\begin{theorem}[Gradient component characterization]
	\label{thm:gradient-lyra-arbs-components}
	A smooth function \(u\) determines a gradient Lyra almost
	Ricci--Bourguignon soliton if and only if the following equations
	hold.
	
	For all \(X,Y\in\Gamma(\mathcal H)\),
	\begin{align}
		&\operatorname{Ric}_{1}(X,Y)
		-
		n_{2}
		\left[
		\operatorname{Hess}_{1}k(X,Y)
		+
		X(k)Y(k)
		\right]
		\nonumber\\
		&\quad
		+
		\operatorname{Hess}_{1}u(X,Y)
		-
		(n-2)\operatorname{Hess}_{1}\sigma(X,Y)
		\nonumber\\
		&\quad
		+
		(n-2)X(\sigma)Y(\sigma)
		-
		X(u)Y(\sigma)
		-
		X(\sigma)Y(u)
		\nonumber\\
		&=
		\left[
		e^{2\sigma}\Lambda
		-
		\Psi_{u,\sigma}
		\right]
		g_{1}(X,Y).
		\label{eq:gradient-lyra-arbs-horizontal}
	\end{align}
	
	For all \(U,V\in\Gamma(\mathcal V)\),
	\begin{align}
		&\operatorname{Ric}(U,V)
		+
		\operatorname{Hess}u(U,V)
		-
		(n-2)\operatorname{Hess}\sigma(U,V)
		\nonumber\\
		&\quad
		+
		(n-2)U(\sigma)V(\sigma)
		-
		U(u)V(\sigma)
		-
		U(\sigma)V(u)
		\nonumber\\
		&=
		f^{2}
		\left[
		e^{2\sigma}\Lambda
		-
		\Psi_{u,\sigma}
		\right]
		g_{2}(U,V).
		\label{eq:gradient-lyra-arbs-vertical}
	\end{align}
	
	For all
	\(X\in\Gamma(\mathcal H)\) and
	\(U\in\Gamma(\mathcal V)\),
	\begin{align}
		0
		={}&
		-
		(n_{2}-1)XU(k)
		+
		XU(u)
		-
		X(k)U(u)
		\nonumber\\
		&-
		(n-2)
		\left[
		XU(\sigma)
		-
		X(k)U(\sigma)
		\right]
		\nonumber\\
		&+
		(n-2)X(\sigma)U(\sigma)
		-
		X(u)U(\sigma)
		-
		X(\sigma)U(u).
		\label{eq:gradient-lyra-arbs-mixed}
	\end{align}
\end{theorem}

\begin{proof}
	Equation~\eqref{eq:gradient-lyra-arbs-conformal} is evaluated
	separately on horizontal, vertical, and mixed vector fields.
	
	For horizontal fields, use
	\eqref{eq:twisted-ricci-horizontal} together with
	\[
	\operatorname{Hess}u(X,Y)
	=
	\operatorname{Hess}_{1}u(X,Y)
	\]
	and
	\[
	\operatorname{Hess}\sigma(X,Y)
	=
	\operatorname{Hess}_{1}\sigma(X,Y).
	\]
	This gives \eqref{eq:gradient-lyra-arbs-horizontal}.
	
	For vertical fields, the required Ricci and Hessian terms are
	given by
	\eqref{eq:twisted-ricci-vertical} and
	\eqref{eq:twisted-hessian-vertical}. Substitution into
	\eqref{eq:gradient-lyra-arbs-conformal} yields
	\eqref{eq:gradient-lyra-arbs-vertical}.
	
	For mixed fields, the metric term vanishes. Using
	\eqref{eq:twisted-ricci-mixed} and
	\eqref{eq:twisted-hessian-mixed}, we obtain
	\begin{align*}
		0
		={}&
		-
		(n_{2}-1)XU(k)
		+
		XU(u)-X(k)U(u)
		\\
		&-
		(n-2)
		\left[
		XU(\sigma)-X(k)U(\sigma)
		\right]
		\\
		&+
		(n-2)X(\sigma)U(\sigma)
		-
		X(u)U(\sigma)
		-
		X(\sigma)U(u),
	\end{align*}
	which is \eqref{eq:gradient-lyra-arbs-mixed}.
	
	Conversely, the three equations determine
	\eqref{eq:gradient-lyra-arbs-conformal} on every pair of tangent
	vectors and hence imply
	\eqref{eq:gradient-lyra-arbs}.
\end{proof}

\begin{corollary}[Gradient trace identity]
	\label{cor:gradient-lyra-arbs-trace}
	Every gradient Lyra almost Ricci--Bourguignon soliton satisfies
	\begin{equation}
		(1-n\rho)R^{L}
		+
		\Delta^{L}u
		=
		n\lambda.
		\label{eq:gradient-lyra-arbs-trace}
	\end{equation}
	Equivalently,
	\begin{align}
		(1-n\rho)R^{L}
		+
		e^{-2\sigma}
		\left[
		\Delta u
		+
		(n-2)
		\left\langle
		\nabla\sigma,\nabla u
		\right\rangle_{g}
		\right]
		=
		n\lambda.
		\label{eq:gradient-lyra-arbs-trace-expanded}
	\end{align}
\end{corollary}

\begin{proof}
	Taking the \(g_{L}\)-trace of
	\eqref{eq:gradient-lyra-arbs} gives
	\[
	R^{L}
	+
	\Delta^{L}u
	=
	n\lambda+n\rho R^{L},
	\]
	which proves \eqref{eq:gradient-lyra-arbs-trace}. The expanded
	form follows directly from
	Equation~\eqref{eq:lyra-laplacian-function}.
\end{proof}

\begin{corollary}[Compact integral identity]
	\label{cor:compact-gradient-lyra-arbs}
	Let \(M\) be compact, oriented, and without boundary. Then every
	gradient Lyra almost Ricci--Bourguignon soliton satisfies
	\begin{equation}
		(1-n\rho)
		\int_{M}
		R^{L}\,d\mu_{L}
		=
		n
		\int_{M}
		\lambda\,d\mu_{L}.
		\label{eq:compact-gradient-integral}
	\end{equation}
	In particular, if \(\lambda\) is constant and
	\(\rho\neq 1/n\), then
	\begin{equation}
		\frac{1}{\operatorname{Vol}_{L}(M)}
		\int_{M}
		R^{L}\,d\mu_{L}
		=
		\frac{n\lambda}{1-n\rho}.
		\label{eq:average-lyra-scalar-curvature}
	\end{equation}
\end{corollary}

\begin{proof}
	Integrating \eqref{eq:gradient-lyra-arbs-trace} over \(M\) and
	using the divergence theorem gives
	\[
	\int_{M}
	\Delta^{L}u\,d\mu_{L}
	=
	0.
	\]
	Hence \eqref{eq:compact-gradient-integral} follows. If
	\(\lambda\) is constant and \(\rho\neq 1/n\), division by
	\((1-n\rho)\operatorname{Vol}_{L}(M)\) gives
	\eqref{eq:average-lyra-scalar-curvature}.
\end{proof}

The systems
\eqref{eq:lyra-arbs-horizontal}--
\eqref{eq:lyra-arbs-mixed} and
\eqref{eq:gradient-lyra-arbs-horizontal}--
\eqref{eq:gradient-lyra-arbs-mixed} provide the basic equations for
the remainder of the paper. In the next section, the mixed
conditions will be used to determine when the soliton structure is
inherited by the factor manifolds and when the Lyra scale forces,
prevents, or compensates for the separation of the twisting
function.

\section{Factor Inheritance and Scale--Twisting Rigidity}
\label{sec:factor-inheritance-rigidity}

The decomposition obtained in
Theorem~\ref{thm:lyra-arbs-component-characterization} shows that
the horizontal and vertical equations determine the structures
induced on the factor manifolds, whereas the mixed equation
\eqref{eq:lyra-arbs-mixed} controls the compatibility between the
twisting function and the Lyra scale. In the ordinary setting, the
mixed component frequently forces the twisting function to become
multiplicatively separable; see
\cite{ElsharkawyTawfiq2026,ElsharkawyCesaranoWanas2026}.
The purpose of this section is to determine how the Lyra scale
modifies this phenomenon.

Throughout the section, let
\[
M=M_{1}\times_{f}M_{2},
\qquad
g_{L}=e^{2\sigma}
\left(
g_{1}\oplus f^{2}g_{2}
\right),
\]
and let
\[
Z=Z_{1}+Z_{2}
\]
be a projectable potential vector field. We continue to write
\[
k=\ln f,
\qquad
n_{1}=\dim M_{1},
\qquad
n_{2}=\dim M_{2},
\qquad
n=n_{1}+n_{2}.
\]

For a symmetric \((0,2)\)-tensor \(S\) on an \(m\)-dimensional
pseudo-Riemannian manifold \((N,h)\), define its trace-free part by
\begin{equation}
	\operatorname{tf}_{h}S
	=
	S
	-
	\frac{1}{m}
	\left(
	\operatorname{tr}_{h}S
	\right)h.
	\label{eq:trace-free-part}
\end{equation}
Thus, \(S\) is proportional to \(h\) if and only if
\[
\operatorname{tf}_{h}S=0.
\]

\subsection{Inheritance by the base manifold}
\label{subsec:base-inheritance}

We first consider the base-adapted case
\begin{equation}
	\sigma=\sigma_{1},
	\qquad
	k=k_{1}+k_{2},
	\label{eq:base-adapted-data}
\end{equation}
where
\[
\sigma_{1},k_{1}\in C^{\infty}(M_{1}),
\qquad
k_{2}\in C^{\infty}(M_{2}).
\]
The corresponding Lyra metric on the base is
\begin{equation}
	g_{L,1}
	=
	e^{2\sigma_{1}}g_{1}.
	\label{eq:base-lyra-metric}
\end{equation}
Let
\[
\operatorname{Ric}^{L,1},
\qquad
R^{L,1}
\]
denote the Ricci tensor and scalar curvature of
\(g_{L,1}\).

Define the symmetric tensor
\begin{align}
	\mathcal B_{1}
	={}&
	\operatorname{Hess}_{1}k_{1}
	+
	dk_{1}\otimes dk_{1}
	+
	\operatorname{Hess}_{1}\sigma_{1}
	-
	d\sigma_{1}\otimes d\sigma_{1},
	\label{eq:base-inheritance-tensor}
\end{align}
and the scalar function
\begin{equation}
	b_{1}
	=
	\left\langle
	\nabla_{1}k_{1},
	\nabla_{1}\sigma_{1}
	\right\rangle_{g_{1}}
	+
	\left\lvert
	\nabla_{1}\sigma_{1}
	\right\rvert_{g_{1}}^{2}.
	\label{eq:base-inheritance-scalar}
\end{equation}

\begin{lemma}
	\label{lem:base-lyra-ricci-splitting}
	Under \eqref{eq:base-adapted-data}, the horizontal component of
	the total Lyra Ricci tensor satisfies
	\begin{align}
		\operatorname{Ric}^{L}(X,Y)
		={}&
		\operatorname{Ric}^{L,1}(X,Y)
		-
		n_{2}\mathcal B_{1}(X,Y)
		-
		n_{2}b_{1}g_{1}(X,Y)
		\label{eq:base-lyra-ricci-splitting}
	\end{align}
	for all \(X,Y\in\Gamma(\mathcal H)\).
\end{lemma}

\begin{proof}
	Using \eqref{eq:lyra-ricci-transformation} and
	\eqref{eq:twisted-ricci-horizontal}, we obtain
	\begin{align}
		\operatorname{Ric}^{L}(X,Y)
		={}&
		\operatorname{Ric}_{1}(X,Y)
		-
		n_{2}
		\left[
		\operatorname{Hess}_{1}k_{1}(X,Y)
		+
		X(k_{1})Y(k_{1})
		\right]
		\nonumber\\
		&-
		(n-2)
		\left[
		\operatorname{Hess}_{1}\sigma_{1}(X,Y)
		-
		X(\sigma_{1})Y(\sigma_{1})
		\right]
		-
		\Theta_{\sigma_{1}}g_{1}(X,Y).
		\label{eq:base-total-ricci-expanded}
	\end{align}
	By \eqref{eq:twisted-laplacian},
	\begin{equation}
		\Delta\sigma_{1}
		=
		\Delta_{1}\sigma_{1}
		+
		n_{2}
		\left\langle
		\nabla_{1}k_{1},
		\nabla_{1}\sigma_{1}
		\right\rangle_{g_{1}},
		\label{eq:base-scale-laplacian}
	\end{equation}
	while
	\[
	\lvert\nabla\sigma_{1}\rvert_{g}^{2}
	=
	\lvert\nabla_{1}\sigma_{1}\rvert_{g_{1}}^{2}.
	\]
	Therefore,
	\begin{align}
		\Theta_{\sigma_{1}}
		={}&
		\Delta_{1}\sigma_{1}
		+
		(n_{1}-2)
		\lvert\nabla_{1}\sigma_{1}\rvert_{g_{1}}^{2}
		+
		n_{2}b_{1}.
		\label{eq:base-theta-splitting}
	\end{align}
	
	On the other hand, the conformal transformation formula on the
	\(n_{1}\)-dimensional base gives
	\begin{align}
		\operatorname{Ric}^{L,1}
		={}&
		\operatorname{Ric}_{1}
		-
		(n_{1}-2)
		\left(
		\operatorname{Hess}_{1}\sigma_{1}
		-
		d\sigma_{1}\otimes d\sigma_{1}
		\right)
		\nonumber\\
		&-
		\left[
		\Delta_{1}\sigma_{1}
		+
		(n_{1}-2)
		\lvert\nabla_{1}\sigma_{1}\rvert_{g_{1}}^{2}
		\right]g_{1}.
		\label{eq:base-intrinsic-lyra-ricci}
	\end{align}
	Substituting \eqref{eq:base-theta-splitting} into
	\eqref{eq:base-total-ricci-expanded} and comparing with
	\eqref{eq:base-intrinsic-lyra-ricci} yields
	\eqref{eq:base-lyra-ricci-splitting}.
\end{proof}

\begin{theorem}[Base inheritance criterion]
	\label{thm:base-inheritance}
	Assume \eqref{eq:base-adapted-data}. For each fixed
	\(y\in M_{2}\), the base manifold
	\[
	\left(
	M_{1},
	g_{L,1},
	Z_{1}
	\right)
	\]
	inherits a Lyra almost Ricci--Bourguignon soliton structure with
	parameter \(\rho\) if and only if
	\begin{equation}
		\operatorname{tf}_{g_{1}}\mathcal B_{1}
		=
		0.
		\label{eq:base-inheritance-condition}
	\end{equation}
	
	More precisely, if
	\begin{equation}
		\mathcal B_{1}
		=
		\beta_{1}g_{1},
		\label{eq:base-inheritance-beta}
	\end{equation}
	then the inherited soliton function is
	\begin{align}
		\lambda_{1,y}
		={}&
		\lambda
		+
		\rho
		\left(
		R^{L}-R^{L,1}
		\right)
		+
		n_{2}e^{-2\sigma_{1}}
		\left(
		\beta_{1}+b_{1}
		\right),
		\label{eq:base-inherited-soliton-function}
	\end{align}
	where the quantities on the right-hand side are evaluated on
	the slice \(M_{1}\times\{y\}\).
	
	A single soliton structure on \(M_{1}\), independent of the
	choice of \(y\), is obtained if and only if the right-hand side
	of \eqref{eq:base-inherited-soliton-function} is independent of
	the fiber variable.
\end{theorem}

\begin{proof}
	Because \(\sigma=\sigma_{1}\), Equation
	\eqref{eq:lyra-lie-horizontal} becomes
	\begin{equation}
		(\mathcal L_{Z}g_{L})(X,Y)
		=
		(\mathcal L_{Z_{1}}g_{L,1})(X,Y).
		\label{eq:base-lie-inheritance}
	\end{equation}
	Using Lemma~\ref{lem:base-lyra-ricci-splitting} in the
	horizontal soliton equation gives
	\begin{align}
		\operatorname{Ric}^{L,1}
		+
		\frac{1}{2}
		\mathcal L_{Z_{1}}g_{L,1}
		={}&
		\left(
		\lambda+\rho R^{L}
		\right)g_{L,1}
		+
		n_{2}\mathcal B_{1}
		+
		n_{2}b_{1}g_{1}.
		\label{eq:base-induced-equation}
	\end{align}
	
	The left-hand side of \eqref{eq:base-induced-equation} defines
	an almost Ricci--Bourguignon soliton on the base precisely when
	the additional tensor on the right-hand side is proportional
	to \(g_{L,1}\), or equivalently to \(g_{1}\). Since
	\(b_{1}g_{1}\) is already pure trace, this is equivalent to
	\[
	\operatorname{tf}_{g_{1}}\mathcal B_{1}=0.
	\]
	This proves \eqref{eq:base-inheritance-condition}.
	
	If \(\mathcal B_{1}=\beta_{1}g_{1}\), then
	\eqref{eq:base-induced-equation} becomes
	\begin{align*}
		\operatorname{Ric}^{L,1}
		+
		\frac{1}{2}
		\mathcal L_{Z_{1}}g_{L,1}
		={}&
		\left[
		\lambda
		+
		\rho R^{L}
		+
		n_{2}e^{-2\sigma_{1}}
		\left(
		\beta_{1}+b_{1}
		\right)
		\right]g_{L,1}.
	\end{align*}
	Writing the coefficient as
	\(\lambda_{1,y}+\rho R^{L,1}\) yields
	\eqref{eq:base-inherited-soliton-function}.
\end{proof}

\begin{corollary}
	\label{cor:base-inheritance-vanishing}
	Under the hypotheses of Theorem~\ref{thm:base-inheritance},
	if
	\begin{equation}
		\operatorname{Hess}_{1}k_{1}
		+
		dk_{1}\otimes dk_{1}
		+
		\operatorname{Hess}_{1}\sigma_{1}
		-
		d\sigma_{1}\otimes d\sigma_{1}
		=
		0,
		\label{eq:base-defect-vanishing}
	\end{equation}
	then \(M_{1}\) inherits a Lyra almost
	Ricci--Bourguignon soliton with
	\begin{equation}
		\lambda_{1,y}
		=
		\lambda
		+
		\rho
		\left(
		R^{L}-R^{L,1}
		\right)
		+
		n_{2}e^{-2\sigma_{1}}b_{1}.
		\label{eq:base-inheritance-vanishing-function}
	\end{equation}
\end{corollary}

\subsection{Inheritance by the fiber slices}
\label{subsec:fiber-inheritance}

We now consider the fiber-adapted setting
\begin{equation}
	\sigma=\sigma_{2},
	\qquad
	k=k_{1}+k_{2},
	\label{eq:fiber-adapted-data}
\end{equation}
where
\[
k_{1}\in C^{\infty}(M_{1}),
\qquad
k_{2},\sigma_{2}\in C^{\infty}(M_{2}).
\]
Set
\begin{equation}
	\omega_{2}
	=
	k_{2}+\sigma_{2},
	\qquad
	\widehat g_{2}
	=
	e^{2\omega_{2}}g_{2}.
	\label{eq:fiber-effective-metric}
\end{equation}
For each fixed \(x\in M_{1}\), the metric induced on the fiber
slice \(\{x\}\times M_{2}\) is
\begin{equation}
	h_{2,x}
	=
	e^{2k_{1}(x)}\widehat g_{2}.
	\label{eq:fiber-slice-metric}
\end{equation}

We use the notation
\begin{equation}
	dk_{2}\odot d\sigma_{2}
	=
	dk_{2}\otimes d\sigma_{2}
	+
	d\sigma_{2}\otimes dk_{2}.
	\label{eq:symmetric-product-definition}
\end{equation}
Define
\begin{align}
	\mathcal C_{2}
	={}&
	\operatorname{Hess}_{2}\sigma_{2}
	-
	dk_{2}\odot d\sigma_{2}
	-
	d\sigma_{2}\otimes d\sigma_{2},
	\label{eq:fiber-inheritance-tensor}
\end{align}
together with
\begin{equation}
	c_{2}
	=
	\left\langle
	\nabla_{2}k_{2},
	\nabla_{2}\sigma_{2}
	\right\rangle_{g_{2}}
	+
	\left\lvert
	\nabla_{2}\sigma_{2}
	\right\rvert_{g_{2}}^{2},
	\label{eq:fiber-inheritance-scalar}
\end{equation}
and
\begin{equation}
	a_{1}
	=
	\Delta_{1}k_{1}
	+
	n_{2}
	\left\lvert
	\nabla_{1}k_{1}
	\right\rvert_{g_{1}}^{2}.
	\label{eq:fiber-base-scalar}
\end{equation}

\begin{lemma}
	\label{lem:fiber-lyra-ricci-splitting}
	Under \eqref{eq:fiber-adapted-data}, the vertical component of
	the Lyra Ricci tensor satisfies
	\begin{align}
		\operatorname{Ric}^{L}(U,V)
		={}&
		\operatorname{Ric}_{\widehat g_{2}}(U,V)
		-
		n_{1}\mathcal C_{2}(U,V)
		\nonumber\\
		&-
		\left[
		n_{1}c_{2}
		+
		f^{2}a_{1}
		\right]
		g_{2}(U,V)
		\label{eq:fiber-lyra-ricci-splitting}
	\end{align}
	for all \(U,V\in\Gamma(\mathcal V)\).
\end{lemma}

\begin{proof}
	Using \eqref{eq:twisted-ricci-vertical} with
	\(k=k_{1}+k_{2}\), we obtain
	\begin{align}
		\operatorname{Ric}(U,V)
		={}&
		\operatorname{Ric}_{2}(U,V)
		-
		(n_{2}-2)
		\left[
		\operatorname{Hess}_{2}k_{2}(U,V)
		-
		U(k_{2})V(k_{2})
		\right]
		\nonumber\\
		&-
		\left[
		\Delta_{2}k_{2}
		+
		(n_{2}-2)
		\lvert\nabla_{2}k_{2}\rvert_{g_{2}}^{2}
		+
		f^{2}a_{1}
		\right]g_{2}(U,V).
		\label{eq:fiber-ordinary-ricci-separated}
	\end{align}
	
	By \eqref{eq:twisted-hessian-vertical},
	\begin{align}
		\operatorname{Hess}\sigma_{2}(U,V)
		={}&
		\operatorname{Hess}_{2}\sigma_{2}(U,V)
		-
		U(k_{2})V(\sigma_{2})
		-
		V(k_{2})U(\sigma_{2})
		\nonumber\\
		&+
		\left\langle
		\nabla_{2}k_{2},
		\nabla_{2}\sigma_{2}
		\right\rangle_{g_{2}}
		g_{2}(U,V).
		\label{eq:fiber-scale-hessian}
	\end{align}
	Moreover, Equations~\eqref{eq:twisted-laplacian} and
	\eqref{eq:twisted-gradient-norm} imply
	\begin{align}
		f^{2}\Theta_{\sigma_{2}}
		={}&
		\Delta_{2}\sigma_{2}
		+
		(n_{2}-2)
		\left\langle
		\nabla_{2}k_{2},
		\nabla_{2}\sigma_{2}
		\right\rangle_{g_{2}}
		\nonumber\\
		&+
		(n-2)
		\lvert\nabla_{2}\sigma_{2}\rvert_{g_{2}}^{2}.
		\label{eq:fiber-theta-splitting}
	\end{align}
	
	On the other hand, the Ricci tensor of
	\(\widehat g_{2}=e^{2\omega_{2}}g_{2}\) is
	\begin{align}
		\operatorname{Ric}_{\widehat g_{2}}
		={}&
		\operatorname{Ric}_{2}
		-
		(n_{2}-2)
		\left(
		\operatorname{Hess}_{2}\omega_{2}
		-
		d\omega_{2}\otimes d\omega_{2}
		\right)
		\nonumber\\
		&-
		\left[
		\Delta_{2}\omega_{2}
		+
		(n_{2}-2)
		\lvert\nabla_{2}\omega_{2}\rvert_{g_{2}}^{2}
		\right]g_{2}.
		\label{eq:fiber-effective-ricci}
	\end{align}
	
	Substituting
	\eqref{eq:fiber-ordinary-ricci-separated}--
	\eqref{eq:fiber-theta-splitting} into
	\eqref{eq:lyra-ricci-transformation}, and then comparing the
	result with \eqref{eq:fiber-effective-ricci}, yields
	\eqref{eq:fiber-lyra-ricci-splitting}.
\end{proof}

\begin{theorem}[Fiber-slice inheritance criterion]
	\label{thm:fiber-inheritance}
	Assume \eqref{eq:fiber-adapted-data}. For each fixed
	\(x\in M_{1}\), the fiber slice
	\[
	\left(
	M_{2},
	h_{2,x},
	Z_{2}
	\right)
	\]
	inherits an almost Ricci--Bourguignon soliton structure with
	parameter \(\rho\) if and only if
	\begin{equation}
		\operatorname{tf}_{g_{2}}\mathcal C_{2}
		=
		0.
		\label{eq:fiber-inheritance-condition}
	\end{equation}
	
	If
	\begin{equation}
		\mathcal C_{2}
		=
		\gamma_{2}g_{2},
		\label{eq:fiber-inheritance-gamma}
	\end{equation}
	then the inherited soliton function is
	\begin{align}
		\lambda_{2,x}
		={}&
		\lambda
		+
		\rho
		\left(
		R^{L}-R_{h_{2,x}}
		\right)
		-
		Z_{1}(k_{1})
		\nonumber\\
		&+
		e^{-2\sigma_{2}}f^{-2}
		\left[
		n_{1}\gamma_{2}
		+
		n_{1}c_{2}
		+
		f^{2}a_{1}
		\right].
		\label{eq:fiber-inherited-soliton-function}
	\end{align}
\end{theorem}

\begin{proof}
	By \eqref{eq:lyra-lie-vertical} and
	\eqref{eq:fiber-slice-metric}, we have
	\begin{align}
		(\mathcal L_{Z}g_{L})(U,V)
		={}&
		(\mathcal L_{Z_{2}}h_{2,x})(U,V)
		+
		2Z_{1}(k_{1})h_{2,x}(U,V).
		\label{eq:fiber-lie-inheritance}
	\end{align}
	Using Lemma~\ref{lem:fiber-lyra-ricci-splitting} in the
	vertical part of \eqref{eq:lyra-arbs-equation} yields
	\begin{align}
		\operatorname{Ric}_{h_{2,x}}
		+
		\frac{1}{2}
		\mathcal L_{Z_{2}}h_{2,x}
		={}&
		\left[
		\lambda+\rho R^{L}
		-
		Z_{1}(k_{1})
		\right]h_{2,x}
		\nonumber\\
		&+
		n_{1}\mathcal C_{2}
		+
		\left[
		n_{1}c_{2}
		+
		f^{2}a_{1}
		\right]g_{2}.
		\label{eq:fiber-induced-equation}
	\end{align}
	Here we used the fact that
	\[
	\operatorname{Ric}_{h_{2,x}}
	=
	\operatorname{Ric}_{\widehat g_{2}},
	\]
	because \(h_{2,x}\) differs from \(\widehat g_{2}\) by a
	constant homothety on each fixed fiber slice.
	
	The last term in \eqref{eq:fiber-induced-equation} is already
	proportional to \(g_{2}\), and hence to \(h_{2,x}\). Therefore,
	the right-hand side is proportional to \(h_{2,x}\) if and only
	if
	\[
	\operatorname{tf}_{g_{2}}\mathcal C_{2}=0.
	\]
	If \(\mathcal C_{2}=\gamma_{2}g_{2}\), then
	\[
	g_{2}
	=
	e^{-2\sigma_{2}}f^{-2}h_{2,x}.
	\]
	Substitution into \eqref{eq:fiber-induced-equation} and
	comparison with
	\[
	\operatorname{Ric}_{h_{2,x}}
	+
	\frac{1}{2}\mathcal L_{Z_{2}}h_{2,x}
	=
	\left(
	\lambda_{2,x}
	+
	\rho R_{h_{2,x}}
	\right)h_{2,x}
	\]
	gives \eqref{eq:fiber-inherited-soliton-function}.
\end{proof}

\begin{remark}
	\label{rem:genuine-fiber-inheritance}
	Theorem~\ref{thm:fiber-inheritance} gives an inherited structure
	on each fiber slice. If \(k_{1}\) is constant, then all the
	metrics \(h_{2,x}\) differ by the same constant homothety, and
	the result determines a single intrinsic soliton structure on
	the fiber manifold. If \(k_{1}\) is nonconstant, the inherited
	structure is naturally a family parameterized by the base.
\end{remark}

\subsection{Rigidity of the mixed soliton equation}
\label{subsec:mixed-rigidity}

The mixed equation
\eqref{eq:lyra-arbs-mixed} may be written as
\begin{align}
	(n_{2}-1)XU(k)
	+
	(n-2)
	\left[
	XU(\sigma)
	-
	X(k)U(\sigma)
	-
	X(\sigma)U(\sigma)
	\right]
	=
	0.
	\label{eq:mixed-rigidity-equation}
\end{align}
Unlike the horizontal and vertical equations,
\eqref{eq:mixed-rigidity-equation} does not involve
\(\lambda\), \(\rho\), or the projectable potential field. It is
therefore an intrinsic condition on the pair \((f,\varphi)\).

\begin{theorem}[Base-dependent scales preserve ordinary rigidity]
	\label{thm:base-scale-rigidity}
	Assume
	\[
	n_{2}>1
	\]
	and suppose that the Lyra scale is base-dependent,
	\begin{equation}
		\sigma=\sigma_{1}.
		\label{eq:base-dependent-scale}
	\end{equation}
	Then every Lyra almost Ricci--Bourguignon soliton satisfies
	\begin{equation}
		XU(k)=0
		\label{eq:ordinary-mixed-separation}
	\end{equation}
	for all horizontal \(X\) and vertical \(U\).
	
	Consequently, on every connected product neighborhood,
	\begin{equation}
		k=k_{1}+k_{2},
		\label{eq:additive-separation-k}
	\end{equation}
	and hence
	\begin{equation}
		f=f_{1}f_{2}.
		\label{eq:multiplicative-separation-f}
	\end{equation}
\end{theorem}

\begin{proof}
	Under \eqref{eq:base-dependent-scale},
	\[
	U(\sigma)=0,
	\qquad
	XU(\sigma)=0.
	\]
	Therefore, Equation~\eqref{eq:mixed-rigidity-equation} reduces
	to
	\[
	(n_{2}-1)XU(k)=0.
	\]
	Since \(n_{2}>1\), Equation
	\eqref{eq:ordinary-mixed-separation} follows.
	
	The vanishing of all mixed derivatives implies that the
	horizontal differential \(d_{1}k\) is independent of the fiber
	variable. Hence, locally,
	\[
	d_{1}k=d k_{1}
	\]
	for some function \(k_{1}\) on \(M_{1}\). It follows that
	\(k-k_{1}\) is constant along \(M_{1}\), and therefore depends
	only on \(M_{2}\). Thus,
	\[
	k=k_{1}+k_{2}.
	\]
	Exponentiating gives
	\[
	f=e^{k_{1}}e^{k_{2}}=f_{1}f_{2}.
	\]
\end{proof}

\begin{corollary}
	\label{cor:ordinary-rigidity-recovery}
	If the Lyra scale is constant and \(n_{2}>1\), then every Lyra
	almost Ricci--Bourguignon soliton reduces locally to a
	multiplicatively separable twisted product.
\end{corollary}

\begin{proof}
	A constant Lyra scale is a special case of
	Theorem~\ref{thm:base-scale-rigidity}.
\end{proof}

The preceding theorem shows that a Lyra scale depending only on
the base cannot preserve a genuinely nonseparable twisting
function. A different phenomenon occurs when the Lyra scale
varies along the fiber.

\begin{theorem}[Fiber-scale compensation normal form]
	\label{thm:fiber-scale-compensation}
	Assume
	\[
	n_{2}>1
	\]
	and let the Lyra scale depend only on the fiber:
	\begin{equation}
		\sigma=\sigma_{2}.
		\label{eq:fiber-dependent-scale}
	\end{equation}
	Set
	\begin{equation}
		\alpha
		=
		\frac{n-2}{n_{2}-1}.
		\label{eq:compensation-exponent}
	\end{equation}
	Then the mixed soliton equation is equivalent to
	\begin{equation}
		U\!\left(
		e^{-\alpha\sigma_{2}}X(k)
		\right)
		=
		0
		\label{eq:weighted-horizontal-derivative}
	\end{equation}
	for all horizontal \(X\) and vertical \(U\).
	
	Consequently, on every simply connected product neighborhood,
	the twisting function has the local normal form
	\begin{equation}
		k(x,y)
		=
		e^{\alpha\sigma_{2}(y)}k_{1}(x)
		+
		k_{2}(y),
		\label{eq:compensated-k-normal-form}
	\end{equation}
	or equivalently,
	\begin{equation}
		f(x,y)
		=
		f_{2}(y)
		\exp
		\left(
		e^{\alpha\sigma_{2}(y)}k_{1}(x)
		\right).
		\label{eq:compensated-f-normal-form}
	\end{equation}
\end{theorem}

\begin{proof}
	Under \eqref{eq:fiber-dependent-scale},
	\[
	X(\sigma_{2})=0,
	\qquad
	XU(\sigma_{2})=0.
	\]
	Hence \eqref{eq:mixed-rigidity-equation} becomes
	\begin{equation}
		(n_{2}-1)U(X(k))
		-
		(n-2)X(k)U(\sigma_{2})
		=
		0.
		\label{eq:fiber-scale-mixed-reduced}
	\end{equation}
	Dividing by \(n_{2}-1\) and using
	\eqref{eq:compensation-exponent}, we obtain
	\[
	U(X(k))
	-
	\alpha X(k)U(\sigma_{2})
	=
	0.
	\]
	Therefore,
	\begin{align*}
		U\!\left(
		e^{-\alpha\sigma_{2}}X(k)
		\right)
		&=
		e^{-\alpha\sigma_{2}}
		\left[
		U(X(k))
		-
		\alpha X(k)U(\sigma_{2})
		\right]
		\\
		&=0,
	\end{align*}
	which proves \eqref{eq:weighted-horizontal-derivative}.
	
	It follows that
	\[
	e^{-\alpha\sigma_{2}}d_{1}k
	\]
	is independent of the fiber variable. Hence there exists a
	one-form \(\eta_{1}\) on \(M_{1}\) such that
	\begin{equation}
		d_{1}k
		=
		e^{\alpha\sigma_{2}}\eta_{1}.
		\label{eq:horizontal-k-one-form}
	\end{equation}
	Since \(d_{1}k\) is exact and
	\(e^{\alpha\sigma_{2}}\) is constant along the base,
	\(\eta_{1}\) is closed. On a simply connected neighborhood,
	there exists \(k_{1}\in C^{\infty}(M_{1})\) such that
	\[
	\eta_{1}=dk_{1}.
	\]
	Equation~\eqref{eq:horizontal-k-one-form} then gives
	\[
	d_{1}
	\left(
	k-e^{\alpha\sigma_{2}}k_{1}
	\right)
	=
	0.
	\]
	Therefore,
	\[
	k-e^{\alpha\sigma_{2}}k_{1}
	=
	k_{2}
	\]
	for some function \(k_{2}\) on \(M_{2}\). This proves
	\eqref{eq:compensated-k-normal-form}. Exponentiation gives
	\eqref{eq:compensated-f-normal-form}.
\end{proof}

\begin{remark}
	\label{rem:genuine-compensation}
	If \(\sigma_{2}\) is nonconstant and \(k_{1}\) is nonconstant,
	then \eqref{eq:compensated-k-normal-form} is generally not of
	the additive form \(k=k_{1}+k_{2}\). Therefore,
	\eqref{eq:compensated-f-normal-form} provides a genuinely
	twisted structure that is permitted by the Lyra scale but
	excluded in the constant-scale case. This is the
	\emph{scale--twisting compensation phenomenon}.
\end{remark}

\begin{corollary}
	\label{cor:fiber-scale-constant-reduction}
	Under the hypotheses of
	Theorem~\ref{thm:fiber-scale-compensation}, if
	\(\sigma_{2}\) is constant, then
	\eqref{eq:compensated-k-normal-form} reduces, after rescaling
	\(k_{1}\), to
	\[
	k=k_{1}+k_{2}.
	\]
\end{corollary}

\subsection{Rigidity under separated scale and twisting data}
\label{subsec:separated-rigidity}

We next suppose that both logarithmic functions are additively
separated:
\begin{equation}
	k=k_{1}+k_{2},
	\qquad
	\sigma=\sigma_{1}+\sigma_{2}.
	\label{eq:fully-separated-data}
\end{equation}

\begin{theorem}[Scale--twisting compensation law]
	\label{thm:scale-twisting-compensation-law}
	Assume \(n>2\) and
	\eqref{eq:fully-separated-data}. Then the mixed soliton
	equation is equivalent to
	\begin{equation}
		X(k_{1}+\sigma_{1})\,U(\sigma_{2})
		=
		0
		\label{eq:separated-compensation-equation}
	\end{equation}
	for all
	\[
	X\in\Gamma(TM_{1}),
	\qquad
	U\in\Gamma(TM_{2}).
	\]
	
	Consequently, on every connected open set on which
	\(d\sigma_{2}\neq0\),
	\begin{equation}
		d(k_{1}+\sigma_{1})=0.
		\label{eq:base-compensation-differential}
	\end{equation}
	Equivalently,
	\begin{equation}
		k_{1}+\sigma_{1}
		=
		C,
		\label{eq:base-compensation-constant}
	\end{equation}
	and hence
	\begin{equation}
		f_{1}\varphi_{1}
		=
		e^{C}.
		\label{eq:scale-twisting-product-constant}
	\end{equation}
\end{theorem}

\begin{proof}
	Under \eqref{eq:fully-separated-data},
	\[
	XU(k)=0,
	\qquad
	XU(\sigma)=0,
	\]
	while
	\[
	X(k)=X(k_{1}),
	\qquad
	X(\sigma)=X(\sigma_{1}),
	\qquad
	U(\sigma)=U(\sigma_{2}).
	\]
	Therefore, Equation~\eqref{eq:mixed-rigidity-equation} reduces
	to
	\begin{align*}
		0
		&=
		-(n-2)
		\left[
		X(k_{1})U(\sigma_{2})
		+
		X(\sigma_{1})U(\sigma_{2})
		\right]
		\\
		&=
		-(n-2)
		X(k_{1}+\sigma_{1})U(\sigma_{2}).
	\end{align*}
	Since \(n>2\), this is equivalent to
	\eqref{eq:separated-compensation-equation}.
	
	If \(d\sigma_{2}\neq0\) on a connected open set, then at each
	point there exists a vertical field \(U\) satisfying
	\[
	U(\sigma_{2})\neq0.
	\]
	Equation~\eqref{eq:separated-compensation-equation} therefore
	implies
	\[
	X(k_{1}+\sigma_{1})=0
	\]
	for every horizontal field \(X\). Hence
	\[
	d(k_{1}+\sigma_{1})=0,
	\]
	which proves \eqref{eq:base-compensation-differential} and
	\eqref{eq:base-compensation-constant}. Exponentiating gives
	\[
	e^{k_{1}}e^{\sigma_{1}}
	=
	f_{1}\varphi_{1}
	=
	e^{C}.
	\]
\end{proof}

\begin{corollary}[Rigidity alternative]
	\label{cor:scale-twisting-rigidity-alternative}
	Under the hypotheses of
	Theorem~\ref{thm:scale-twisting-compensation-law}, at least one
	of the following alternatives holds locally:
	\begin{enumerate}
		\item the fiber part of the Lyra scale is constant,
		\[
		d\sigma_{2}=0;
		\]
		\item the base twisting and scale functions compensate each
		other,
		\[
		f_{1}\varphi_{1}=\mathrm{constant}.
		\]
	\end{enumerate}
\end{corollary}

\begin{corollary}[Nonexistence criterion]
	\label{cor:separated-nonexistence}
	Let \(n>2\), and suppose
	\eqref{eq:fully-separated-data} holds on a connected open set.
	If
	\[
	d\sigma_{2}\neq0
	\]
	and
	\[
	d(k_{1}+\sigma_{1})\neq0
	\]
	at some point, then no projectable Lyra almost
	Ricci--Bourguignon soliton exists on that open set.
\end{corollary}

\begin{proof}
	The two nonvanishing conditions contradict
	\eqref{eq:separated-compensation-equation}.
\end{proof}

\begin{corollary}[Rigid normalized form]
	\label{cor:rigid-normalized-form}
	Assume the second alternative in
	Corollary~\ref{cor:scale-twisting-rigidity-alternative}. After
	multiplying \(f_{1}\) or \(\varphi_{1}\) by a positive constant,
	one may normalize
	\begin{equation}
		f_{1}\varphi_{1}=1.
		\label{eq:normalized-compensation}
	\end{equation}
	In this normalization,
	\begin{equation}
		\sigma_{1}=-k_{1}.
		\label{eq:normalized-log-compensation}
	\end{equation}
\end{corollary}

\begin{proof}
	By \eqref{eq:scale-twisting-product-constant},
	\[
	f_{1}\varphi_{1}=e^{C}.
	\]
	Absorbing \(e^{C}\) into either positive factor gives
	\eqref{eq:normalized-compensation}. Taking logarithms yields
	\eqref{eq:normalized-log-compensation}.
\end{proof}

\begin{remark}[Low-dimensional cases]
	\label{rem:low-dimensional-mixed-rigidity}
	The assumptions \(n_{2}>1\) and \(n>2\) are essential in the
	preceding rigidity results.
	
	If \(n_{2}=1\), then the first term in
	\eqref{eq:mixed-rigidity-equation} vanishes, and the twisting
	function is controlled entirely through the Lyra scale terms.
	
	If \(n=2\), then all terms multiplied by \(n-2\) vanish, and
	the mixed equation becomes
	\[
	(n_{2}-1)XU(k)=0.
	\]
	Thus, the scale--twisting compensation phenomenon is specific
	to dimensions greater than two.
\end{remark}

Theorems~\ref{thm:base-inheritance} and
\ref{thm:fiber-inheritance} identify the trace-free obstructions
to factor inheritance, while
Theorems~\ref{thm:fiber-scale-compensation} and
\ref{thm:scale-twisting-compensation-law} describe the two
principal rigidity mechanisms. In particular, a base-dependent
Lyra scale preserves the ordinary separability obstruction,
whereas a fiber-dependent scale permits the nonseparable normal
form \eqref{eq:compensated-f-normal-form}. These results will be
used in the next section to characterize Einstein reductions and
special classes of potential vector fields.

\section{Einstein Metrics and Special Potential Fields}
\label{sec:einstein-special-fields}

This section studies the influence of distinguished potential vector
fields on a Lyra almost Ricci--Bourguignon soliton. In the ordinary
warped and twisted settings, Killing, conformal, homothetic, and
concurrent potentials frequently reduce the soliton equation to an
Einstein-type condition; see
\cite{BinTurkiEtAl2022,ShenawyEtAl2023,
	ElsharkawyTawfiq2026}. Here we establish the corresponding Lyra
results and determine their consequences for the total manifold and
its factor manifolds.

Throughout this section, let
\[
\left(
M=M_{1}\times_{f}M_{2},
g_{L},
Z,
\lambda,
\rho
\right)
\]
be a Lyra almost Ricci--Bourguignon soliton satisfying
\eqref{eq:lyra-arbs-equation}, and set
\[
\Lambda=\lambda+\rho R^{L}
\]
as in \eqref{eq:lyra-arbs-lambda}. Unless otherwise stated, the
total dimension satisfies
\[
n=n_{1}+n_{2}\geq 3.
\]

\subsection{Projectable conformal potential fields}
\label{subsec:projectable-conformal-potentials}

\begin{definition}
	\label{def:lyra-special-vector-fields}
	A vector field \(Z\in\Gamma(TM)\) is called:
	\begin{enumerate}
		\item \emph{Lyra conformal} if
		\begin{equation}
			\mathcal L_{Z}g_{L}
			=
			2\psi g_{L}
			\label{eq:lyra-conformal-field}
		\end{equation}
		for some smooth function
		\(\psi\in C^{\infty}(M)\);
		
		\item \emph{Lyra Killing} if
		\[
		\psi=0;
		\]
		
		\item \emph{Lyra homothetic} if
		\[
		\psi=c
		\]
		for some constant \(c\in\mathbb R\);
		
		\item \emph{Lyra closed conformal} if
		\begin{equation}
			\nabla^{L}_{A}Z
			=
			\psi A
			\label{eq:lyra-closed-conformal-field}
		\end{equation}
		for every \(A\in\Gamma(TM)\);
		
		\item \emph{Lyra concurrent} if
		\begin{equation}
			\nabla^{L}_{A}Z
			=
			A
			\label{eq:lyra-concurrent-field}
		\end{equation}
		for every \(A\in\Gamma(TM)\).
	\end{enumerate}
\end{definition}

A Lyra closed conformal field satisfies
\eqref{eq:lyra-conformal-field}, since
\[
g_{L}(\nabla^{L}_{A}Z,B)
+
g_{L}(A,\nabla^{L}_{B}Z)
=
2\psi g_{L}(A,B).
\]
In particular, every Lyra concurrent field is Lyra homothetic with
conformal factor \(1\).

Let the potential field be projectable:
\begin{equation}
	Z=Z_{1}+Z_{2},
	\label{eq:special-projectable-potential}
\end{equation}
where \(Z_{1}\) and \(Z_{2}\) are lifted vector fields on
\(M_{1}\) and \(M_{2}\), respectively.

\begin{proposition}[Factorwise characterization of conformal fields]
	\label{prop:factorwise-conformal-characterization}
	A projectable vector field \(Z=Z_{1}+Z_{2}\) is Lyra conformal
	with conformal factor \(\psi\) if and only if
	\begin{subequations}
		\label{eq:factorwise-conformal-system}
		\begin{align}
			\mathcal L_{Z_{1}}g_{1}
			&=
			2\psi_{1}g_{1},
			\qquad
			\psi_{1}
			=
			\psi-Z(\sigma),
			\label{eq:base-conformal-factor}
			\\
			\mathcal L_{Z_{2}}g_{2}
			&=
			2\psi_{2}g_{2},
			\qquad
			\psi_{2}
			=
			\psi-Z(\omega),
			\label{eq:fiber-conformal-factor}
		\end{align}
	\end{subequations}
	where
	\[
	\omega=k+\sigma
	\]
	is defined in
	\eqref{eq:combined-scale-twisting-function}.
	
	In this case,
	\begin{equation}
		\psi_{1}-\psi_{2}
		=
		Z(k).
		\label{eq:factor-conformal-difference}
	\end{equation}
\end{proposition}

\begin{proof}
	Suppose first that \(Z\) is Lyra conformal. Evaluating
	\eqref{eq:lyra-conformal-field} on horizontal fields and using
	\eqref{eq:lyra-lie-horizontal}, we obtain
	\[
	e^{2\sigma}
	\left[
	(\mathcal L_{Z_{1}}g_{1})(X,Y)
	+
	2Z(\sigma)g_{1}(X,Y)
	\right]
	=
	2\psi e^{2\sigma}g_{1}(X,Y).
	\]
	After cancelling \(e^{2\sigma}\), this becomes
	\[
	(\mathcal L_{Z_{1}}g_{1})(X,Y)
	=
	2\left[
	\psi-Z(\sigma)
	\right]g_{1}(X,Y),
	\]
	which proves \eqref{eq:base-conformal-factor}.
	
	Similarly, evaluating on vertical fields and using
	\eqref{eq:lyra-lie-vertical} gives
	\[
	e^{2\sigma}f^{2}
	\left[
	(\mathcal L_{Z_{2}}g_{2})(U,V)
	+
	2Z(\omega)g_{2}(U,V)
	\right]
	=
	2\psi e^{2\sigma}f^{2}g_{2}(U,V).
	\]
	Hence,
	\[
	(\mathcal L_{Z_{2}}g_{2})(U,V)
	=
	2\left[
	\psi-Z(\omega)
	\right]g_{2}(U,V),
	\]
	which is \eqref{eq:fiber-conformal-factor}.
	
	The mixed component vanishes automatically by
	\eqref{eq:lyra-lie-mixed}. Conversely,
	\eqref{eq:factorwise-conformal-system} together with
	\eqref{eq:lyra-lie-components} implies
	\[
	\mathcal L_{Z}g_{L}=2\psi g_{L}.
	\]
	
	Finally,
	\[
	\psi_{1}-\psi_{2}
	=
	Z(\omega)-Z(\sigma)
	=
	Z(k),
	\]
	which proves \eqref{eq:factor-conformal-difference}.
\end{proof}

\begin{corollary}[Projectable Lyra Killing fields]
	\label{cor:projectable-lyra-killing}
	A projectable field \(Z=Z_{1}+Z_{2}\) is Lyra Killing if and
	only if
	\begin{subequations}
		\label{eq:factorwise-killing-system}
		\begin{align}
			\mathcal L_{Z_{1}}g_{1}
			&=
			-2Z(\sigma)g_{1},
			\label{eq:base-killing-system}
			\\
			\mathcal L_{Z_{2}}g_{2}
			&=
			-2Z(\omega)g_{2}.
			\label{eq:fiber-killing-system}
		\end{align}
	\end{subequations}
\end{corollary}

\begin{proof}
	Set \(\psi=0\) in
	Proposition~\ref{prop:factorwise-conformal-characterization}.
\end{proof}

\begin{corollary}[Projectable Lyra homothetic fields]
	\label{cor:projectable-lyra-homothetic}
	A projectable field \(Z=Z_{1}+Z_{2}\) is Lyra homothetic with
	factor \(c\) if and only if
	\begin{subequations}
		\label{eq:factorwise-homothetic-system}
		\begin{align}
			\mathcal L_{Z_{1}}g_{1}
			&=
			2\left[
			c-Z(\sigma)
			\right]g_{1},
			\label{eq:base-homothetic-system}
			\\
			\mathcal L_{Z_{2}}g_{2}
			&=
			2\left[
			c-Z(\omega)
			\right]g_{2}.
			\label{eq:fiber-homothetic-system}
		\end{align}
	\end{subequations}
\end{corollary}

\subsection{Conformal potentials and Einstein metrics}
\label{subsec:conformal-einstein-equivalence}

\begin{definition}
	\label{def:lyra-einstein-metric}
	The Lyra metric \(g_{L}\) is called \emph{Einstein} if
	\begin{equation}
		\operatorname{Ric}^{L}
		=
		\mu g_{L}
		\label{eq:lyra-einstein-condition}
	\end{equation}
	for some constant \(\mu\in\mathbb R\).
\end{definition}

\begin{theorem}[Conformal--Einstein equivalence]
	\label{thm:conformal-einstein-equivalence}
	Let
	\[
	\left(
	M,g_{L},Z,\lambda,\rho
	\right)
	\]
	be a Lyra almost Ricci--Bourguignon soliton on a connected
	manifold of dimension \(n\geq3\). Then the following statements
	are equivalent:
	\begin{enumerate}
		\item \(Z\) is Lyra conformal;
		\item \(g_{L}\) is an Einstein metric.
	\end{enumerate}
	
	More precisely, if
	\[
	\mathcal L_{Z}g_{L}=2\psi g_{L},
	\]
	then
	\begin{equation}
		\operatorname{Ric}^{L}
		=
		\mu g_{L},
		\qquad
		\mu
		=
		\lambda+\rho R^{L}-\psi,
		\label{eq:einstein-factor-conformal}
	\end{equation}
	where \(\mu\) is constant. Moreover,
	\begin{equation}
		R^{L}
		=
		n\mu
		\label{eq:einstein-scalar-curvature}
	\end{equation}
	and
	\begin{equation}
		\lambda
		=
		(1-n\rho)\mu+\psi.
		\label{eq:lambda-conformal-einstein}
	\end{equation}
\end{theorem}

\begin{proof}
	Suppose that \(Z\) is Lyra conformal. Substituting
	\[
	\frac12\mathcal L_{Z}g_{L}
	=
	\psi g_{L}
	\]
	into \eqref{eq:lyra-arbs-equation} gives
	\[
	\operatorname{Ric}^{L}
	=
	\left(
	\lambda+\rho R^{L}-\psi
	\right)g_{L}.
	\]
	Define
	\[
	\mu
	=
	\lambda+\rho R^{L}-\psi.
	\]
	Then
	\[
	\operatorname{Ric}^{L}
	=
	\mu g_{L}.
	\]
	Taking the trace yields
	\[
	R^{L}=n\mu.
	\]
	
	The contracted second Bianchi identity gives
	\[
	\operatorname{div}_{L}
	\operatorname{Ric}^{L}
	=
	\frac12 dR^{L}.
	\]
	Since
	\[
	\operatorname{div}_{L}(\mu g_{L})=d\mu
	\]
	and
	\[
	dR^{L}=n\,d\mu,
	\]
	we obtain
	\[
	d\mu
	=
	\frac n2 d\mu.
	\]
	Because \(n\geq3\),
	\[
	\left(
	1-\frac n2
	\right)d\mu=0
	\]
	implies
	\[
	d\mu=0.
	\]
	Thus, \(\mu\) is constant and \(g_{L}\) is Einstein.
	
	Conversely, suppose that
	\[
	\operatorname{Ric}^{L}
	=
	\mu g_{L}.
	\]
	Using \eqref{eq:lyra-arbs-equation}, we find
	\[
	\frac12\mathcal L_{Z}g_{L}
	=
	\left(
	\lambda+\rho R^{L}-\mu
	\right)g_{L}.
	\]
	Hence \(Z\) is Lyra conformal with factor
	\[
	\psi
	=
	\lambda+\rho R^{L}-\mu.
	\]
	
	Finally, substituting
	\[
	R^{L}=n\mu
	\]
	into \eqref{eq:einstein-factor-conformal} gives
	\[
	\mu
	=
	\lambda+n\rho\mu-\psi,
	\]
	which is equivalent to
	\eqref{eq:lambda-conformal-einstein}.
\end{proof}

\begin{remark}
	\label{rem:dimension-two-einstein}
	The restriction \(n\geq3\) is essential for concluding that
	\(\mu\) is constant. In dimension two,
	\[
	\operatorname{Ric}^{L}
	=
	\frac{R^{L}}{2}g_{L}
	\]
	holds identically, but the scalar curvature need not be
	constant.
\end{remark}

\begin{corollary}[Killing potential]
	\label{cor:killing-potential-einstein}
	Let \(n\geq3\). If the potential field \(Z\) is Lyra Killing,
	then \(g_{L}\) is Einstein and
	\begin{equation}
		\lambda
		=
		(1-n\rho)\mu.
		\label{eq:killing-soliton-relation}
	\end{equation}
	Consequently, the almost soliton function \(\lambda\) is
	constant.
	
	If
	\[
	\rho\neq\frac1n,
	\]
	then
	\begin{equation}
		\mu
		=
		\frac{\lambda}{1-n\rho}.
		\label{eq:killing-einstein-factor}
	\end{equation}
	If
	\[
	\rho=\frac1n,
	\]
	then necessarily
	\begin{equation}
		\lambda=0.
		\label{eq:critical-killing-lambda}
	\end{equation}
\end{corollary}

\begin{proof}
	Set \(\psi=0\) in
	\eqref{eq:lambda-conformal-einstein}.
\end{proof}

\begin{corollary}[Homothetic potential]
	\label{cor:homothetic-potential-einstein}
	Let \(n\geq3\). If \(Z\) is Lyra homothetic with constant factor
	\(c\), then \(g_{L}\) is Einstein and
	\begin{equation}
		\lambda
		=
		(1-n\rho)\mu+c.
		\label{eq:homothetic-soliton-relation}
	\end{equation}
	In particular, \(\lambda\) is constant.
	
	If
	\[
	\rho\neq\frac1n,
	\]
	then
	\begin{equation}
		\mu
		=
		\frac{\lambda-c}{1-n\rho}.
		\label{eq:homothetic-einstein-factor}
	\end{equation}
	At the critical value
	\[
	\rho=\frac1n,
	\]
	one necessarily has
	\begin{equation}
		\lambda=c.
		\label{eq:critical-homothetic-lambda}
	\end{equation}
\end{corollary}

\begin{proof}
	Set \(\psi=c\) in
	\eqref{eq:lambda-conformal-einstein}.
\end{proof}

\begin{corollary}[Concurrent potential]
	\label{cor:concurrent-potential-einstein}
	Let \(n\geq3\). If \(Z\) is Lyra concurrent, then \(g_{L}\) is
	Einstein and
	\begin{equation}
		\lambda
		=
		(1-n\rho)\mu+1.
		\label{eq:concurrent-soliton-relation}
	\end{equation}
	Consequently, every Lyra almost Ricci--Bourguignon soliton
	admitting a concurrent potential is, in fact, a Lyra
	Ricci--Bourguignon soliton.
	
	If
	\[
	\rho\neq\frac1n,
	\]
	then
	\begin{equation}
		\mu
		=
		\frac{\lambda-1}{1-n\rho}.
		\label{eq:concurrent-einstein-factor}
	\end{equation}
	For
	\[
	\rho=\frac1n,
	\]
	one has
	\begin{equation}
		\lambda=1.
		\label{eq:critical-concurrent-lambda}
	\end{equation}
\end{corollary}

\begin{proof}
	A concurrent field satisfies
	\[
	\mathcal L_{Z}g_{L}=2g_{L}.
	\]
	The result follows from
	Corollary~\ref{cor:homothetic-potential-einstein} with \(c=1\).
\end{proof}

\begin{corollary}[Closed conformal potential]
	\label{cor:closed-conformal-potential}
	Let \(Z\) satisfy
	\[
	\nabla^{L}_{A}Z=\psi A.
	\]
	Then \(g_{L}\) is Einstein and
	\begin{equation}
		d\lambda=d\psi.
		\label{eq:lambda-psi-differential}
	\end{equation}
	Equivalently,
	\begin{equation}
		\lambda-\psi
		=
		(1-n\rho)\mu
		\label{eq:lambda-psi-constant}
	\end{equation}
	is constant.
\end{corollary}

\begin{proof}
	A Lyra closed conformal field satisfies
	\[
	\mathcal L_{Z}g_{L}=2\psi g_{L}.
	\]
	By Theorem~\ref{thm:conformal-einstein-equivalence},
	\[
	\lambda
	=
	(1-n\rho)\mu+\psi.
	\]
	Since \(\mu\) and \(\rho\) are constant, differentiation gives
	\[
	d\lambda=d\psi.
	\]
\end{proof}

\subsection{Rigidity of conformal soliton potentials}
\label{subsec:conformal-potential-rigidity}

\begin{theorem}[Rigidity of non-almost conformal solitons]
	\label{thm:nonalmost-conformal-rigidity}
	Let
	\[
	\left(
	M,g_{L},Z,\lambda,\rho
	\right)
	\]
	be a connected Lyra Ricci--Bourguignon soliton of dimension
	\(n\geq3\); that is, \(\lambda\) is constant. If \(Z\) is Lyra
	conformal, then \(Z\) is necessarily Lyra homothetic.
	
	More precisely, its conformal factor is the constant
	\begin{equation}
		\psi
		=
		\lambda-(1-n\rho)\mu,
		\label{eq:constant-conformal-factor}
	\end{equation}
	where \(\mu\) is the Einstein constant of \(g_{L}\).
\end{theorem}

\begin{proof}
	By Theorem~\ref{thm:conformal-einstein-equivalence},
	\(g_{L}\) is Einstein and
	\[
	\lambda=(1-n\rho)\mu+\psi.
	\]
	Both \(\lambda\) and \(\mu\) are constant. Therefore,
	\(\psi\) is constant, and hence \(Z\) is Lyra homothetic.
\end{proof}

\begin{theorem}[Compact conformal rigidity]
	\label{thm:compact-conformal-rigidity}
	Let \((M,g_{L})\) be compact, connected, oriented, and
	Riemannian, with \(n\geq3\). Suppose that
	\[
	\left(
	M,g_{L},Z,\lambda,\rho
	\right)
	\]
	is a Lyra Ricci--Bourguignon soliton and \(Z\) is Lyra
	conformal. Then:
	\begin{enumerate}
		\item \(Z\) is Lyra Killing;
		\item \(g_{L}\) is Einstein;
		\item
		\begin{equation}
			\lambda
			=
			(1-n\rho)\mu.
			\label{eq:compact-conformal-relation}
		\end{equation}
	\end{enumerate}
\end{theorem}

\begin{proof}
	By Theorem~\ref{thm:nonalmost-conformal-rigidity}, the conformal
	factor \(\psi\) is constant. Taking the \(g_{L}\)-trace of
	\[
	\mathcal L_{Z}g_{L}=2\psi g_{L}
	\]
	gives
	\[
	\operatorname{div}_{L}Z=n\psi.
	\]
	Integration over the compact manifold yields
	\[
	n\psi\operatorname{Vol}_{L}(M)
	=
	\int_{M}\operatorname{div}_{L}Z\,d\mu_{L}
	=
	0.
	\]
	Thus,
	\[
	\psi=0,
	\]
	and \(Z\) is Lyra Killing. The remaining conclusions follow
	from Corollary~\ref{cor:killing-potential-einstein}.
\end{proof}

\subsection{Gradient conformal potentials}
\label{subsec:gradient-conformal-potentials}

Let
\[
Z=\nabla^{L}u
\]
for some \(u\in C^{\infty}(M)\). Then
\[
\mathcal L_{\nabla^{L}u}g_{L}
=
2\operatorname{Hess}^{L}u.
\]

\begin{theorem}[Gradient conformal--Einstein characterization]
	\label{thm:gradient-conformal-einstein}
	Let
	\[
	\left(
	M,g_{L},\nabla^{L}u,\lambda,\rho
	\right)
	\]
	be a gradient Lyra almost Ricci--Bourguignon soliton on a
	connected manifold of dimension \(n\geq3\). Then \(g_{L}\) is
	Einstein if and only if
	\begin{equation}
		\operatorname{Hess}^{L}u
		=
		\psi g_{L}
		\label{eq:gradient-conformal-hessian}
	\end{equation}
	for some smooth function \(\psi\).
	
	In this case,
	\begin{equation}
		\lambda
		=
		(1-n\rho)\mu+\psi,
		\label{eq:gradient-conformal-relation}
	\end{equation}
	where \(\mu\) is the Einstein constant.
\end{theorem}

\begin{proof}
	The condition
	\[
	\operatorname{Hess}^{L}u=\psi g_{L}
	\]
	is equivalent to
	\[
	\mathcal L_{\nabla^{L}u}g_{L}
	=
	2\psi g_{L}.
	\]
	Thus, \(\nabla^{L}u\) is Lyra conformal. The conclusion follows
	directly from
	Theorem~\ref{thm:conformal-einstein-equivalence}.
\end{proof}

Using \eqref{eq:lyra-hessian-function}, condition
\eqref{eq:gradient-conformal-hessian} can be written entirely in
terms of the twisted metric \(g\) as
\begin{align}
	\operatorname{Hess}u
	&-
	du\otimes d\sigma
	-
	d\sigma\otimes du
	\nonumber\\
	&+
	\left\langle
	\nabla u,\nabla\sigma
	\right\rangle_{g}g
	=
	\psi e^{2\sigma}g.
	\label{eq:gradient-conformal-background-form}
\end{align}
Equation~\eqref{eq:gradient-conformal-background-form} will be
used in the construction of explicit examples.

\begin{corollary}[Affine Lyra potential]
	\label{cor:affine-lyra-potential}
	Suppose
	\[
	\operatorname{Hess}^{L}u=0.
	\]
	Then \(\nabla^{L}u\) is Lyra Killing, \(g_{L}\) is Einstein, and
	\begin{equation}
		\lambda
		=
		(1-n\rho)\mu.
		\label{eq:affine-gradient-relation}
	\end{equation}
\end{corollary}

\begin{proof}
	When \(\operatorname{Hess}^{L}u=0\),
	\[
	\mathcal L_{\nabla^{L}u}g_{L}=0.
	\]
	The result follows from
	Corollary~\ref{cor:killing-potential-einstein}.
\end{proof}

\begin{corollary}[Compact gradient rigidity]
	\label{cor:compact-gradient-conformal-rigidity}
	Let \((M,g_{L})\) be compact, connected, and Riemannian, with
	\(n\geq3\). If
	\[
	\left(
	M,g_{L},\nabla^{L}u,\lambda,\rho
	\right)
	\]
	is a Lyra Ricci--Bourguignon soliton and
	\(\nabla^{L}u\) is Lyra conformal, then \(u\) is constant and
	the soliton is trivial.
\end{corollary}

\begin{proof}
	By Theorem~\ref{thm:compact-conformal-rigidity},
	\(\nabla^{L}u\) is Lyra Killing. Hence,
	\[
	\operatorname{Hess}^{L}u=0.
	\]
	Taking the trace gives
	\[
	\Delta^{L}u=0.
	\]
	A harmonic function on a compact connected Riemannian manifold
	is constant. Therefore,
	\[
	u=\mathrm{constant}
	\]
	and
	\[
	\nabla^{L}u=0.
	\]
\end{proof}

\subsection{Einstein inheritance by the factor manifolds}
\label{subsec:einstein-factor-inheritance}

The factor inheritance criteria obtained in
Section~\ref{sec:factor-inheritance-rigidity} can now be combined
with Theorem~\ref{thm:conformal-einstein-equivalence}.

\begin{theorem}[Einstein reduction on the base]
	\label{thm:base-einstein-reduction}
	Assume the base-adapted conditions
	\eqref{eq:base-adapted-data} and the inheritance condition
	\eqref{eq:base-inheritance-condition}. Let
	\[
	n_{1}\geq3.
	\]
	For each fixed \(y\in M_{2}\), let
	\[
	\lambda_{1,y}
	\]
	be the inherited soliton function given by
	\eqref{eq:base-inherited-soliton-function}. Then the following
	conditions are equivalent:
	\begin{enumerate}
		\item \(g_{L,1}=e^{2\sigma_{1}}g_{1}\) is Einstein;
		\item \(Z_{1}\) is conformal with respect to \(g_{L,1}\).
	\end{enumerate}
	
	If
	\begin{equation}
		\mathcal L_{Z_{1}}g_{L,1}
		=
		2\psi_{1}g_{L,1},
		\label{eq:base-lyra-conformal-field}
	\end{equation}
	then
	\begin{equation}
		\operatorname{Ric}^{L,1}
		=
		\mu_{1}g_{L,1},
		\label{eq:base-lyra-einstein}
	\end{equation}
	where \(\mu_{1}\) is constant on each connected base slice and
	\begin{equation}
		\lambda_{1,y}
		=
		(1-n_{1}\rho)\mu_{1}
		+
		\psi_{1}.
		\label{eq:base-einstein-relation}
	\end{equation}
\end{theorem}

\begin{proof}
	By Theorem~\ref{thm:base-inheritance}, the base satisfies
	\[
	\operatorname{Ric}^{L,1}
	+
	\frac12
	\mathcal L_{Z_{1}}g_{L,1}
	=
	\left(
	\lambda_{1,y}
	+
	\rho R^{L,1}
	\right)g_{L,1}.
	\]
	Applying
	Theorem~\ref{thm:conformal-einstein-equivalence} in dimension
	\(n_{1}\) proves the equivalence and gives
	\eqref{eq:base-einstein-relation}.
\end{proof}

\begin{corollary}
	\label{cor:total-conformal-base-einstein}
	Under the hypotheses of
	Theorem~\ref{thm:base-einstein-reduction}, suppose that the
	total potential \(Z\) is Lyra conformal with factor \(\psi\).
	Then
	\[
	\mathcal L_{Z_{1}}g_{L,1}
	=
	2\psi g_{L,1},
	\]
	and consequently \(g_{L,1}\) is Einstein with
	\begin{equation}
		\lambda_{1,y}
		=
		(1-n_{1}\rho)\mu_{1}
		+
		\psi.
		\label{eq:base-total-conformal-relation}
	\end{equation}
\end{corollary}

\begin{proof}
	Under \(\sigma=\sigma_{1}\), the horizontal identity
	\eqref{eq:base-lie-inheritance} gives
	\[
	(\mathcal L_{Z}g_{L})(X,Y)
	=
	(\mathcal L_{Z_{1}}g_{L,1})(X,Y).
	\]
	The total conformal condition therefore implies
	\[
	\mathcal L_{Z_{1}}g_{L,1}
	=
	2\psi g_{L,1}.
	\]
	The result follows from
	Theorem~\ref{thm:base-einstein-reduction}.
\end{proof}

\begin{theorem}[Einstein reduction on the fiber slices]
	\label{thm:fiber-einstein-reduction}
	Assume the fiber-adapted conditions
	\eqref{eq:fiber-adapted-data} and the inheritance condition
	\eqref{eq:fiber-inheritance-condition}. Let
	\[
	n_{2}\geq3.
	\]
	For each fixed \(x\in M_{1}\), let
	\[
	h_{2,x}
	=
	e^{2k_{1}(x)}\widehat g_{2}
	\]
	be the fiber metric from
	\eqref{eq:fiber-slice-metric}, and let
	\(\lambda_{2,x}\) be given by
	\eqref{eq:fiber-inherited-soliton-function}. Then the following
	conditions are equivalent:
	\begin{enumerate}
		\item \(h_{2,x}\) is Einstein;
		\item \(Z_{2}\) is conformal with respect to \(h_{2,x}\).
	\end{enumerate}
	
	If
	\begin{equation}
		\mathcal L_{Z_{2}}h_{2,x}
		=
		2\psi_{2,x}h_{2,x},
		\label{eq:fiber-slice-conformal-field}
	\end{equation}
	then
	\begin{equation}
		\operatorname{Ric}_{h_{2,x}}
		=
		\mu_{2,x}h_{2,x},
		\label{eq:fiber-slice-einstein}
	\end{equation}
	where
	\begin{equation}
		\lambda_{2,x}
		=
		(1-n_{2}\rho)\mu_{2,x}
		+
		\psi_{2,x}.
		\label{eq:fiber-einstein-relation}
	\end{equation}
\end{theorem}

\begin{proof}
	Theorem~\ref{thm:fiber-inheritance} gives
	\[
	\operatorname{Ric}_{h_{2,x}}
	+
	\frac12\mathcal L_{Z_{2}}h_{2,x}
	=
	\left(
	\lambda_{2,x}
	+
	\rho R_{h_{2,x}}
	\right)h_{2,x}.
	\]
	The conclusion follows by applying
	Theorem~\ref{thm:conformal-einstein-equivalence} in dimension
	\(n_{2}\).
\end{proof}

\begin{corollary}
	\label{cor:total-conformal-fiber-einstein}
	Under the hypotheses of
	Theorem~\ref{thm:fiber-einstein-reduction}, suppose that the
	total potential \(Z\) is Lyra conformal with factor \(\psi\).
	Then
	\begin{equation}
		\mathcal L_{Z_{2}}h_{2,x}
		=
		2
		\left[
		\psi-Z_{1}(k_{1})
		\right]
		h_{2,x}.
		\label{eq:fiber-induced-conformal-factor}
	\end{equation}
	Consequently, \(h_{2,x}\) is Einstein and
	\begin{equation}
		\lambda_{2,x}
		=
		(1-n_{2}\rho)\mu_{2,x}
		+
		\psi
		-
		Z_{1}(k_{1}).
		\label{eq:fiber-total-conformal-relation}
	\end{equation}
\end{corollary}

\begin{proof}
	Using \eqref{eq:fiber-lie-inheritance}, we have
	\[
	\mathcal L_{Z}g_{L}\big|_{\mathcal V\times\mathcal V}
	=
	\mathcal L_{Z_{2}}h_{2,x}
	+
	2Z_{1}(k_{1})h_{2,x}.
	\]
	Since the total field is Lyra conformal,
	\[
	\mathcal L_{Z}g_{L}\big|_{\mathcal V\times\mathcal V}
	=
	2\psi h_{2,x}.
	\]
	Therefore,
	\[
	\mathcal L_{Z_{2}}h_{2,x}
	=
	2
	\left[
	\psi-Z_{1}(k_{1})
	\right]
	h_{2,x}.
	\]
	The result follows from
	Theorem~\ref{thm:fiber-einstein-reduction}.
\end{proof}

\begin{remark}
	\label{rem:critical-factor-parameters}
	At the critical values
	\[
	\rho=\frac{1}{n_{1}}
	\qquad\text{or}\qquad
	\rho=\frac{1}{n_{2}},
	\]
	the corresponding factor equations become
	\[
	\lambda_{1,y}=\psi_{1}
	\qquad\text{or}\qquad
	\lambda_{2,x}=\psi_{2,x}.
	\]
	In these cases, the Einstein constants are not determined by
	the traced soliton equations alone.
\end{remark}

Theorem~\ref{thm:conformal-einstein-equivalence} shows that
conformal potential fields and Einstein Lyra metrics are
equivalent within the almost Ricci--Bourguignon soliton class.
Corollaries~\ref{cor:killing-potential-einstein}--
\ref{cor:concurrent-potential-einstein} give the corresponding
Killing, homothetic, and concurrent reductions, while
Theorems~\ref{thm:base-einstein-reduction} and
\ref{thm:fiber-einstein-reduction} transfer these conclusions to
the inherited factor structures. 

\section{Examples and Nonexistence Results}
\label{sec:examples-nonexistence}

This section illustrates the principal geometric phenomena obtained
in the preceding sections. The first two examples provide exact
Lyra Ricci--Bourguignon solitons with flat and nonflat Einstein
metrics. The third example shows that the Lyra scale can compensate
for a genuinely nonseparable twisting function. We then give local
and global nonexistence consequences of the mixed and traced
soliton equations.

\subsection{Exact Lyra Ricci--Bourguignon solitons}
\label{subsec:exact-lyra-solitons}

\begin{example}[Scale--warping cancellation]
	\label{ex:scale-warping-cancellation}
	
	Let
	\[
	M_{1}\subset\mathbb R^{n_{1}},
	\qquad
	M_{2}=\mathbb R^{n_{2}},
	\]
	with Euclidean metrics \(\overline g_{1}\) and \(g_{2}\).
	For an arbitrary function \(s\in C^{\infty}(M_{1})\), define
	\begin{equation}
		g_{1}=e^{-2s}\overline g_{1},
		\qquad
		f=e^{-s},
		\qquad
		\varphi=e^{s}.
		\label{eq:scale-cancellation-data}
	\end{equation}
	Then
	\[
	k=-s,
	\qquad
	\sigma=s,
	\qquad
	f\varphi=1,
	\]
	and the Lyra metric becomes
	\begin{equation}
		g_{L}
		=
		e^{2s}
		\left(
		g_{1}\oplus f^{2}g_{2}
		\right)
		=
		\overline g_{1}\oplus g_{2}.
		\label{eq:scale-cancellation-metric}
	\end{equation}
	Hence,
	\begin{equation}
		\operatorname{Ric}^{L}=0,
		\qquad
		R^{L}=0.
		\label{eq:scale-cancellation-curvature}
	\end{equation}
	
	Let
	\begin{equation}
		Z
		=
		c
		\left(
		\sum_{i=1}^{n_{1}}
		x_{i}\frac{\partial}{\partial x_{i}}
		+
		\sum_{\alpha=1}^{n_{2}}
		y_{\alpha}\frac{\partial}{\partial y_{\alpha}}
		\right),
		\label{eq:scale-cancellation-potential}
	\end{equation}
	where \(c\in\mathbb R\). Then
	\[
	\mathcal L_{Z}g_{L}=2c\,g_{L},
	\]
	and therefore
	\[
	\left(
	M,g_{L},Z,\lambda=c,\rho
	\right)
	\]
	is a Lyra Ricci--Bourguignon soliton for every
	\(\rho\in\mathbb R\).
\end{example}

\begin{proof}
	Using \eqref{eq:scale-cancellation-curvature}, the soliton
	equation \eqref{eq:lyra-arbs-equation} reduces to
	\[
	\frac12\mathcal L_{Z}g_{L}
	=
	\lambda g_{L}.
	\]
	Since \(Z\) is homothetic with factor \(c\), the equation holds
	for \(\lambda=c\).
\end{proof}

\begin{remark}
	\label{rem:scale-cancellation-significance}
	
	The condition \(f\varphi=1\) is precisely the normalized
	scale--twisting compensation
	\eqref{eq:normalized-compensation}. Moreover,
	\[
	\mathcal B_{1}=0,
	\qquad
	b_{1}=0,
	\]
	so the base inheritance defects in
	\eqref{eq:base-inheritance-tensor} and
	\eqref{eq:base-inheritance-scalar} vanish identically.
	Thus, this example simultaneously realizes scale compensation,
	factor inheritance, and the homothetic Einstein reduction.
\end{remark}

\begin{example}[A nonflat Einstein Lyra soliton]
	\label{ex:hyperbolic-lyra-soliton}
	
	Let
	\[
	M_{1}=\mathbb R^{n-1},
	\qquad
	M_{2}=\mathbb R,
	\qquad
	n\geq3,
	\]
	with coordinates
	\((x_{1},\ldots,x_{n-1},t)\), and define
	\begin{equation}
		g_{1}
		=
		\sum_{i=1}^{n-1}dx_{i}^{2},
		\qquad
		g_{2}=dt^{2},
		\qquad
		f=e^{-at},
		\qquad
		\varphi=e^{at},
		\label{eq:hyperbolic-data}
	\end{equation}
	where \(a\neq0\). Then
	\begin{equation}
		g_{L}
		=
		dt^{2}
		+
		e^{2at}
		\sum_{i=1}^{n-1}dx_{i}^{2}.
		\label{eq:hyperbolic-lyra-metric}
	\end{equation}
	This metric has constant sectional curvature \(-a^{2}\), and
	hence
	\begin{equation}
		\operatorname{Ric}^{L}
		=
		-(n-1)a^{2}g_{L},
		\qquad
		R^{L}
		=
		-n(n-1)a^{2}.
		\label{eq:hyperbolic-lyra-curvature}
	\end{equation}
	
	For the Killing field
	\[
	Z=\frac{\partial}{\partial x_{1}},
	\]
	the soliton equation holds with
	\begin{equation}
		\lambda
		=
		(n\rho-1)(n-1)a^{2}.
		\label{eq:hyperbolic-soliton-function}
	\end{equation}
\end{example}

\begin{proof}
	Since \(\mathcal L_{Z}g_{L}=0\), the assertion follows directly
	from Corollary~\ref{cor:killing-potential-einstein} with
	Einstein constant
	\[
	\mu=-(n-1)a^{2}.
	\]
\end{proof}

\subsection{A genuinely twisted scale-compensated family}
\label{subsec:genuine-scale-compensation}

The next example demonstrates the main distinction between the
ordinary and Lyra settings.

\begin{example}[A nonseparable mixed-compatible family]
	\label{ex:genuine-scale-compensation}
	
	Let
	\[
	M_{1}=\mathbb R,
	\qquad
	M_{2}=\mathbb R^{q},
	\qquad
	q>1,
	\]
	with Euclidean metrics and coordinates
	\[
	x,
	\qquad
	(y_{1},\ldots,y_{q}).
	\]
	Since \(n_{1}=1\) and \(n_{2}=q\), the exponent in
	\eqref{eq:compensation-exponent} is
	\[
	\alpha
	=
	\frac{n-2}{n_{2}-1}
	=
	1.
	\]
	
	Choose constants \(a,b\in\mathbb R\) and
	\(\kappa\in C^{\infty}(M_{2})\), and define
	\begin{equation}
		\sigma=ay_{1},
		\qquad
		k
		=
		bx\,e^{ay_{1}}
		+
		\kappa(y).
		\label{eq:genuine-compensation-data}
	\end{equation}
	Equivalently,
	\begin{equation}
		\varphi=e^{ay_{1}},
		\qquad
		f
		=
		\exp
		\left(
		bx\,e^{ay_{1}}+\kappa(y)
		\right).
		\label{eq:genuine-compensation-functions}
	\end{equation}
	
	For \(X=\partial_{x}\),
	\begin{equation}
		e^{-\sigma}X(k)
		=
		b.
		\label{eq:genuine-compensation-identity}
	\end{equation}
	Hence,
	\[
	U\!\left(
	e^{-\sigma}X(k)
	\right)=0
	\]
	for every vertical field \(U\). By
	Theorem~\ref{thm:fiber-scale-compensation}, the mixed soliton
	equation \eqref{eq:lyra-arbs-mixed} is satisfied.
\end{example}

If \(ab\neq0\), then
\begin{equation}
	\frac{\partial^{2}k}
	{\partial x\,\partial y_{1}}
	=
	ab e^{ay_{1}}
	\neq0.
	\label{eq:genuine-compensation-nonseparable}
\end{equation}
Therefore, \(k\) is not additively separable and \(f\) is genuinely
twisted. In contrast, when the Lyra scale is constant, the ordinary
mixed condition requires \(XU(k)=0\). Thus,
\eqref{eq:genuine-compensation-data} explicitly shows that the Lyra
scale permits nonseparable twisting through the weighted
compensation law.

The remaining horizontal and vertical equations
\eqref{eq:lyra-arbs-horizontal} and
\eqref{eq:lyra-arbs-vertical} determine the admissible potential
field and soliton function for any complete soliton built from this
family.

\subsection{Local rigidity and nonexistence}
\label{subsec:local-rigidity-nonexistence}

\begin{proposition}[Local mixed obstruction]
	\label{prop:local-mixed-obstruction}
	
	Let
	\[
	M_{1}=\mathbb R,
	\qquad
	M_{2}=\mathbb R^{q},
	\qquad
	q>1,
	\]
	and consider
	\begin{equation}
		k=xy_{1},
		\qquad
		\sigma=s(x),
		\label{eq:local-obstruction-data}
	\end{equation}
	where \(s\in C^{\infty}(\mathbb R)\). Then no Lyra almost
	Ricci--Bourguignon soliton with projectable potential field
	exists on any product neighborhood.
\end{proposition}

\begin{proof}
	The scale is base-dependent. Therefore,
	Theorem~\ref{thm:base-scale-rigidity} requires
	\[
	XU(k)=0.
	\]
	However, for
	\[
	X=\frac{\partial}{\partial x},
	\qquad
	U=\frac{\partial}{\partial y_{1}},
	\]
	we have
	\[
	XU(k)=1.
	\]
	This contradicts the mixed soliton equation.
\end{proof}

\begin{remark}
	The obstruction in
	Proposition~\ref{prop:local-mixed-obstruction} is independent
	of \(\lambda\), \(\rho\), and the projectable potential field.
	Thus, it cannot be removed by changing the soliton type or its
	potential.
\end{remark}

\subsection{A compact global obstruction}
\label{subsec:compact-global-obstruction}

\begin{theorem}[Compact sign obstruction]
	\label{thm:compact-sign-obstruction}
	
	Let
	\[
	\left(
	M,g_{L},\nabla^{L}u,\lambda,\rho
	\right)
	\]
	be a gradient Lyra almost Ricci--Bourguignon soliton on a
	compact, connected, oriented Riemannian manifold without
	boundary. Then the following cases are impossible:
	\begin{enumerate}
		\item
		\[
		(1-n\rho)R^{L}\leq0
		\quad\text{and}\quad
		\lambda>0;
		\]
		\item
		\[
		(1-n\rho)R^{L}\geq0
		\quad\text{and}\quad
		\lambda<0.
		\]
	\end{enumerate}
\end{theorem}

\begin{proof}
	Using the compact identity
	\eqref{eq:compact-gradient-integral}, we have
	\[
	(1-n\rho)
	\int_{M}R^{L}\,d\mu_{L}
	=
	n
	\int_{M}\lambda\,d\mu_{L}.
	\]
	In the first case, the left-hand side is nonpositive while the
	right-hand side is positive. In the second case, the left-hand
	side is nonnegative while the right-hand side is negative.
	Both cases are impossible.
\end{proof}

\begin{corollary}
	\label{cor:critical-parameter-nonexistence}
	
	At the critical value
	\[
	\rho=\frac1n,
	\]
	every compact gradient Lyra almost Ricci--Bourguignon soliton
	satisfies
	\begin{equation}
		\int_{M}\lambda\,d\mu_{L}=0.
		\label{eq:critical-zero-mean}
	\end{equation}
	Consequently, no such soliton exists when \(\lambda\) has a
	strict constant sign.
\end{corollary}

The exact models above show that the Lyra scale may cancel a
nonconstant warping function and produce flat or nonflat Einstein
solitons. More importantly,
Example~\ref{ex:genuine-scale-compensation} demonstrates that the
scale can preserve genuinely twisted structures that are excluded
in the ordinary setting. Proposition
\ref{prop:local-mixed-obstruction} and
Theorem~\ref{thm:compact-sign-obstruction} show that the same
compatibility equations also yield effective local and global
nonexistence criteria.

\section{Conclusion}
\label{sec:conclusion}

This paper defined Lyra almost Ricci--Bourguignon solitons on
twisted warped product manifolds and established their component,
inheritance, rigidity, and Einstein characterizations. The main
contribution is the identification of the Lyra scale as an active
geometric mechanism that can preserve, restrict, or compensate for
the twisting structure. The examples and nonexistence results confirm
that this interaction produces genuinely new behavior beyond the
ordinary twisted-product setting.

\section*{Acknowledgements}

The authors extend their appreciation to Princess Nourah Bint Abdulrahman University for funding this research under Researchers Supporting Project number (PNURSP2026R895), Princess Nourah Bint Abdulrahman University, Riyadh, Saudi Arabia.

\section*{Declarations}

\textbf{Data Availability:} No new data were created or analyzed in this study.

\textbf{Conflicts of Interest:} The authors declare that they have no known financial conflicts of interest or personal relationships that could have influenced the work presented in this paper.

\textbf{Funding:} This work was supported and funded by Princess Nourah Bint Abdulrahman University for funding this research under Researchers Supporting Project number (PNURSP2026R895), Princess Nourah Bint Abdulrahman University, Riyadh, Saudi Arabia.

\textbf{Ethics Approval:} The authors hereby affirm that the contents of this manuscript are original. Furthermore, it has neither been published elsewhere in any language fully or partly, nor is it under review for publication anywhere.

\vspace{5mm}
\vspace{2mm}
	

\begin{thebibliography}{99}
	
	\bibitem{Lyra1951}
	G. Lyra,
	``\"{U}ber eine Modifikation der Riemannschen Geometrie,''
	\emph{Mathematische Zeitschrift},
	vol. 54, pp. 52--64, 1951.
	doi: 10.1007/BF01175135.
	
	\bibitem{SenDunn1971}
	D. K. Sen and K. A. Dunn,
	``A scalar-tensor theory of gravitation in a modified Riemannian
	manifold,''
	\emph{Journal of Mathematical Physics},
	vol. 12, no. 4, pp. 578--586, 1971.
	doi: 10.1063/1.1665623.
	
	\bibitem{CuzinattoMoraisPimentel2021}
	R. R. Cuzinatto, E. M. de Morais, and B. M. Pimentel,
	``Lyra scalar-tensor theory: A scalar-tensor theory of gravity on
	Lyra manifold,''
	\emph{Physical Review D},
	vol. 103, Art. 124002, 2021.
	doi: 10.1103/PhysRevD.103.124002.
	
	\bibitem{BertinEtAl2025}
	M. C. Bertin, R. R. Cuzinatto, J. A. Paquiyauri, and B. M. Pimentel,
	``The Lyra--Schwarzschild spacetime,''
	\emph{Universe},
	vol. 11, no. 9, Art. 315, 2025.
	doi: 10.3390/universe11090315.
	
	\bibitem{ValadaoSobreroBergliaffa2026}
	E. C. Valad\~{a}o, F. Sobrero, and S. E. Perez Bergliaffa,
	``Scalar-tensor theories in the Lyra geometry: Invariance under local
	transformations of length units and the Jordan--Einstein frame
	conundrum,''
	\emph{General Relativity and Gravitation},
	vol. 58, Art. 6, 2026.
	doi: 10.1007/s10714-025-03509-8.
	
	\bibitem{BishopONeill1969}
	R. L. Bishop and B. O'Neill,
	``Manifolds of negative curvature,''
	\emph{Transactions of the American Mathematical Society},
	vol. 145, pp. 1--49, 1969.
	doi: 10.1090/S0002-9947-1969-0251664-4.
	
	\bibitem{ONeill1983}
	B. O'Neill,
	\emph{Semi-Riemannian Geometry with Applications to Relativity}.
	Academic Press, New York, 1983.
	
	\bibitem{PongeReckziegel1993}
	R. Ponge and H. Reckziegel,
	``Twisted products in pseudo-Riemannian geometry,''
	\emph{Geometriae Dedicata},
	vol. 48, pp. 15--25, 1993.
	doi: 10.1007/BF01265674.
	
	\bibitem{FernandezLopezEtAl2001}
	M. Fern\'{a}ndez-L\'{o}pez, E. Garc\'{\i}a-R\'{\i}o,
	D. N. Kupeli, and B. \"{U}nal,
	``A curvature condition for a twisted product to be a warped product,''
	\emph{Manuscripta Mathematica},
	vol. 106, no. 2, pp. 213--217, 2001.
	doi: 10.1007/s002290100204.
	
	\bibitem{Unal2001}
	B. \"{U}nal,
	``Doubly warped products,''
	\emph{Differential Geometry and its Applications},
	vol. 15, no. 3, pp. 253--263, 2001.
	doi: 10.1016/S0926-2245(01)00051-1.
	
	\bibitem{Besse2008}
	A. L. Besse,
	\emph{Einstein Manifolds}.
	Springer, Berlin, 2008,
	reprint of the 1987 edition.
	
	\bibitem{Chen2017}
	B.-Y. Chen,
	\emph{Differential Geometry of Warped Product Manifolds and
		Submanifolds}.
	World Scientific, Hackensack, 2017.
	
	\bibitem{ElsharkawyEtAl2025Pseudo}
	A. Elsharkawy, H. K. Elsayied, A. M. Tawfiq, and F. Alghamdi,
	``Geometric analysis of the pseudo-projective curvature tensor in
	doubly and twisted warped product manifolds,''
	\emph{AIMS Mathematics},
	vol. 10, no. 1, pp. 56--71, 2025.
	doi: 10.3934/math.2025004.
	
	\bibitem{ElsayiedTawfiqElsharkawy2025}
	H. K. Elsayied, A. M. Tawfiq, and A. Elsharkawy,
	``Mixed doubly sequential warped product manifolds,''
	\emph{Physica Scripta},
	vol. 100, no. 5, Art. 055229, 2025.
	doi: 10.1088/1402-4896/adcdd1.
	
	\bibitem{Hamilton1982}
	R. S. Hamilton,
	``Three-manifolds with positive Ricci curvature,''
	\emph{Journal of Differential Geometry},
	vol. 17, no. 2, pp. 255--306, 1982.
	doi: 10.4310/jdg/1214436922.
	
	\bibitem{Bourguignon1981}
	J.-P. Bourguignon,
	``Ricci curvature and Einstein metrics,''
	in \emph{Global Differential Geometry and Global Analysis},
	Lecture Notes in Mathematics, vol. 838,
	Springer, Berlin, pp. 42--63, 1981.
	doi: 10.1007/BFb0088841.
	
	\bibitem{CatinoEtAl2017}
	G. Catino, L. Cremaschi, Z. Djadli, C. Mantegazza, and L. Mazzieri,
	``The Ricci--Bourguignon flow,''
	\emph{Pacific Journal of Mathematics},
	vol. 287, no. 2, pp. 337--370, 2017.
	doi: 10.2140/pjm.2017.287.337.
	
	\bibitem{Dwivedi2021}
	S. Dwivedi,
	``Some results on Ricci--Bourguignon solitons and almost solitons,''
	\emph{Canadian Mathematical Bulletin},
	vol. 64, no. 3, pp. 591--604, 2021.
	doi: 10.4153/S0008439520000673.
	
	\bibitem{BlagaTastan2021}
	A. M. Blaga and H. M. Ta\c{s}tan,
	``Some results on almost $\eta$-Ricci--Bourguignon solitons,''
	\emph{Journal of Geometry and Physics},
	vol. 168, Art. 104316, 2021.
	doi: 10.1016/j.geomphys.2021.104316.
	
	\bibitem{Ghosh2022}
	A. Ghosh,
	``Certain triviality results for Ricci--Bourguignon almost solitons,''
	\emph{Journal of Geometry and Physics},
	vol. 182, Art. 104681, 2022.
	doi: 10.1016/j.geomphys.2022.104681.
	
	\bibitem{BinTurkiEtAl2022}
	N. Bin Turki, S. Shenawy, H. K. El-Sayied, N. Syied,
	and C. A. Mantica,
	``$\rho$-Einstein solitons on warped product manifolds and
	applications,''
	\emph{Journal of Mathematics},
	vol. 2022, Art. 1028339, 2022.
	doi: 10.1155/2022/1028339.
	
	\bibitem{ShenawyEtAl2023}
	S. Shenawy, N. Bin Turki, N. Syied, and C. A. Mantica,
	``Almost Ricci--Bourguignon solitons on doubly warped product
	manifolds,''
	\emph{Universe},
	vol. 9, no. 9, Art. 396, 2023.
	doi: 10.3390/universe9090396.
	
	\bibitem{PahanDutta2023}
	S. Pahan and S. Dutta,
	``Characterization of sequential warped product gradient
	Ricci--Bourguignon soliton,''
	\emph{Filomat},
	vol. 37, no. 27, pp. 9273--9285, 2023.
	doi: 10.2298/FIL2327273P.
	
	\bibitem{KayaOzgur2024}
	D. A\c{c}{\i}kg\"{o}z Kaya and C. \"{O}zg\"{u}r,
	``Ricci--Bourguignon solitons on sequential warped product
	manifolds,''
	\emph{Journal of Mathematical Physics, Analysis, Geometry},
	vol. 20, no. 2, pp. 205--220, 2024.
	doi: 10.15407/mag20.02.205.
	
	\bibitem{GulerUnal2025}
	S. G\"{u}ler and B. \"{U}nal,
	``Gradient $\rho$-Einstein solitons and applications,''
	\emph{Facta Universitatis, Series: Mathematics and Informatics},
	vol. 40, no. 1, pp. 233--248, 2025.
	doi: 10.22190/FUMI240813018G.
	
	\bibitem{AlSodaisEtAl2025}
	H. Al-Sodais, N. Bin Turki, S. Deshmukh, B.-Y. Chen,
	and H. M. Shah,
	``Some new characterizations of trivial Ricci--Bourguignon solitons,''
	\emph{Journal of Mathematics},
	vol. 2025, Art. 7917018, 2025.
	doi: 10.1155/jom/7917018.
	
	\bibitem{ElsharkawyEtAl2025Yamabe}
	A. Elsharkawy, H. E. Semary, U. C. De, and E. F. Wanas,
	``Almost quasi-Yamabe solitons and gradient almost quasi-Yamabe
	solitons on twisted and doubly warped product manifolds,''
	\emph{The European Physical Journal Plus},
	vol. 140, Art. 1162, 2025.
	doi: 10.1140/epjp/s13360-025-07042-0.
	
	\bibitem{ElsharkawyTawfiq2026}
	A. Elsharkawy and A. M. Tawfiq,
	``Ricci--Bourguignon solitons and Einstein metrics on twisted warped
	product manifolds,''
	\emph{Mathematical Methods in the Applied Sciences},
	vol. 49, no. 2, pp. 841--850, 2026.
	doi: 10.1002/mma.70194.
	
	\bibitem{SemaryEtAl2026}
	H. E. Semary, A. Aldukeel, B.-Y. Chen, U. C. De,
	E. F. Wanas, and A. Elsharkawy,
	``$\rho$-Einstein solitons on doubly warped product manifolds and
	applications,''
	\emph{Chinese Journal of Physics},
	vol. 101, pp. 758--772, 2026.
	doi: 10.1016/j.cjph.2026.03.002.
	
	\bibitem{ElsharkawyCesaranoWanas2026}
	A. Elsharkawy, C. Cesarano, and E. F. Wanas,
	``Rigidity and structural constraints for $\rho$-Einstein solitons
	on twisted warped product manifolds,''
	\emph{AIMS Mathematics},
	vol. 11, no. 6, pp. 16288--16304, 2026.
	doi: 10.3934/math.2026669.
	
\end{thebibliography}
\end{document}